\documentclass[12pt, a4paper]{amsart}

\usepackage[textwidth=125mm, textheight=195mm, includefoot, twoside]{geometry}
\usepackage[bookmarksopen=true]{hyperref}

\usepackage{amsfonts,amssymb,verbatim}
\usepackage{latexsym}
\usepackage{mathrsfs}
\usepackage{stmaryrd}
\usepackage{xspace}
\usepackage{enumerate, paralist}
\usepackage{graphicx}
\usepackage[all]{xy}
\usepackage{extarrows}

\usepackage{txfonts, pxfonts}

\usepackage{amsthm}
\usepackage{amsmath}

\newtheorem{thm}{Theorem}[section]
 \newtheorem{cor}[thm]{Corollary}
 \newtheorem{lem}[thm]{Lemma}
 \newtheorem{prop}[thm]{Proposition}

 \theoremstyle{definition}
  \newtheorem{defn}[thm]{Definition}
 \theoremstyle{remark}
 \newtheorem{rem}[thm]{Remark}
  \newtheorem{ex}[thm]{Example}

 \def\rank{\mathrm{rank}\,}
 \def\ml{\langle}
 \def\mr{\rangle}
 \def\ql{[}
 \def\qr{]}
 \def\m{\overline{\mathfrak{m}}}
 \def\M{\overline{\mathfrak{M}}}
\def\kappao{{\kappa_0}}
\def\kappav{{\kappa_5}}
\let\kappavv=\kappav
\def\kappaiv{{\kappa_4}}
\def\R{\mathbb R}
\begin{document}

\title{A four point characterisation for coarse median spaces}

\author{Graham A. Niblo}
\author{Nick Wright}
\author{Jiawen Zhang}
\address{School of Mathematics, University of Southampton, Highfield, SO17 1BJ, United Kingdom.}
\email{\{g.a.niblo,n.j.wright,jiawen.zhang\}@soton.ac.uk}

\date{}

\thanks{Partially supported by the Sino-British Fellowship Trust by Royal Society.}
\baselineskip=16pt

\begin{abstract}
Coarse median spaces simultaneously generalise the classes of hyperbolic spaces and median algebras, and arise naturally in the study of the mapping class groups and many other contexts. Their definition as originally conceived by Bowditch requires median approximations for all finite subsets of the space. Here we provide a simplification of the definition in the form of a $4$-point condition analogous to Gromov's $4$-point condition defining hyperbolicity. We give an intrinsic characterisation of rank in terms of the coarse median operator and use this to give a direct proof that rank $1$ geodesic coarse median spaces are $\delta$-hyperbolic, bypassing Bowditch's use of asymptotic cones. A key ingredient of the proof is a new definition of intervals in coarse median spaces and an analysis of their interaction with  geodesics.
\end{abstract}
\maketitle

\section{Introduction}

Coarse median spaces and groups were introduced by Bowditch in 2013 \cite{bowditch2013coarse} as a coarse variant of classical median algebras. The notion of a coarse median group leads to a unified viewpoint on several interesting classes, including Gromov's hyperbolic groups, mapping class groups, and CAT(0) cubical groups. Bowditch showed that geodesic hyperbolic spaces are exactly geodesic coarse median spaces of rank 1, and mapping class groups are examples of coarse median spaces of finite rank \cite{bowditch2013coarse}. In 2014 \cite{behrstock2017hierarchically, behrstock2015hierarchically}, Behrstoke, Hagen and Sisto introduced the notion of hierarchically hyperbolic spaces, and showed that these are coarse median.

Intuitively, a coarse median space $(X,d,\mu)$ is a metric space $(X,d)$ equipped with a ternary operator $\mu$ (called the coarse median), in which every finite subset can be approximated by a finite CAT(0) cube complex, with distortion controlled by the metric. This can be viewed as a wide-ranging extension of Gromov's observation that in a $\delta$-hyperbolic space finite subsets can be approximated by trees. See also Zeidler's Master's thesis \cite{zeidler2013coarse}.

In this paper we simplify the definition of a coarse median space, replacing the requirement to approximate arbitrary finite subsets with a simplified $4$-point condition which may be viewed as a high dimensional analogue of Gromov's $4$-point condition for hyperbolicity. Our condition asserts that given any four points $a,b,c,d$ the two  iterated coarse medians $\mu(\mu(a,b,c),b,d)$ and $\mu(a,b,\mu(c,b,d))$ are uniformly close. As illustrated below this corresponds to an approximation by a CAT(0) cube complex of dimension $3$, where the corresponding iterated medians coincide.

\begin{figure}[h] 
   \centering
   \includegraphics[scale=0.5]{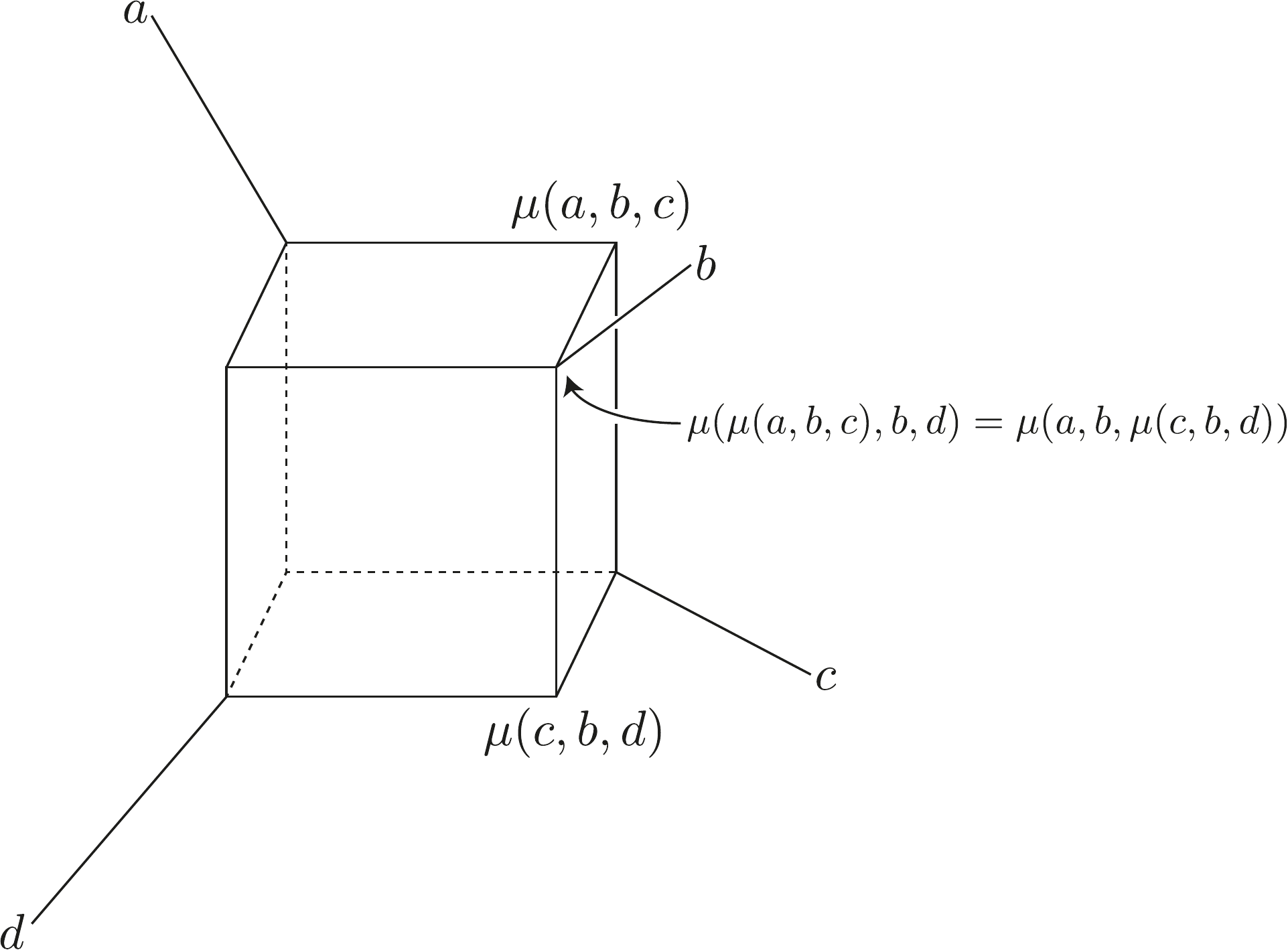}
   \caption{The CAT(0) cube complex associated to the free median algebra on $\{a,b,c,d\}$}
   \label{fig:freemedian}
\end{figure}

We recall that Gromov gave a $4$-point condition characterising hyperbolicity for geodesic spaces which, in essence, asserts that any four points can be approximated by one of the trees shown in Figure \ref{fig:treeapproximations}.
\begin{figure}[h] 
   \centering
   \includegraphics[scale=0.5]{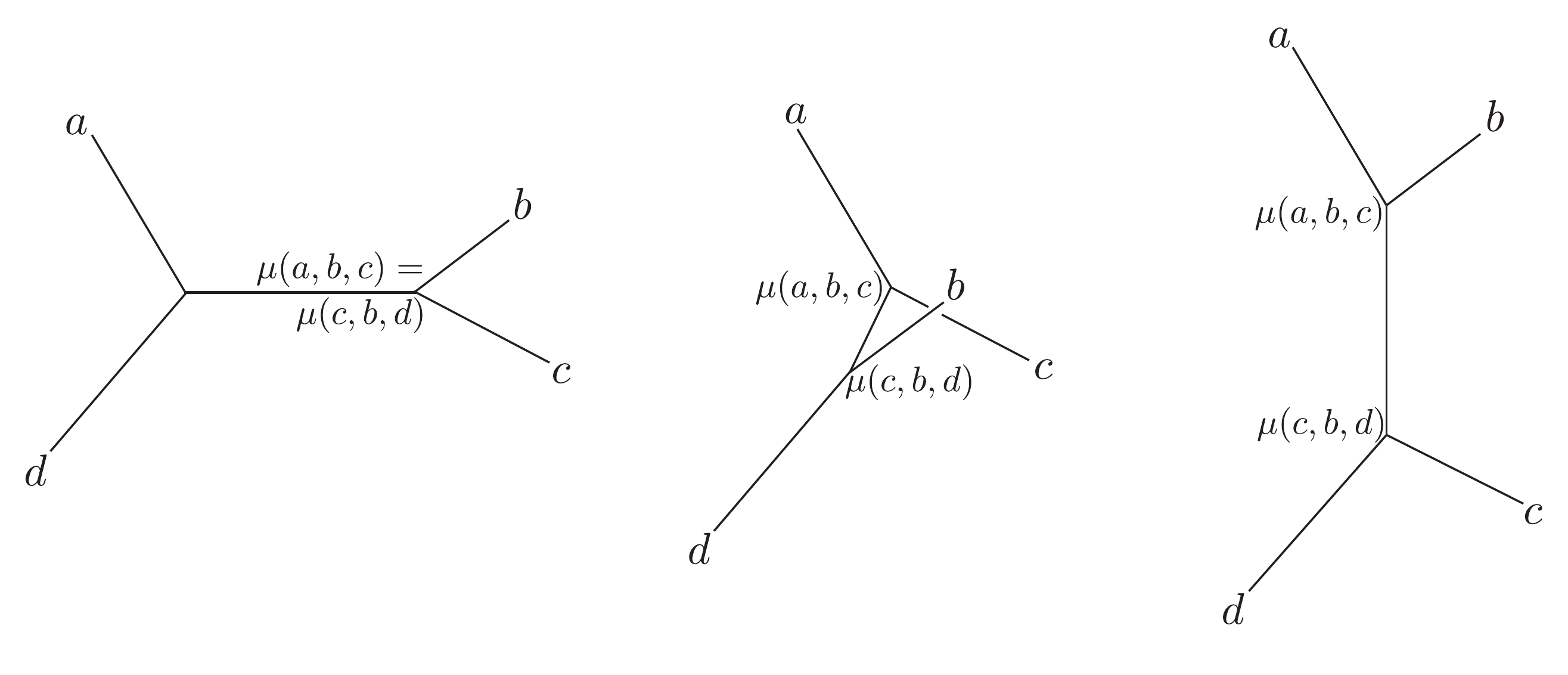}
   \caption{Gromov's $4$-point condition}
   \label{fig:treeapproximations}
\end{figure}
We can visualise each of these trees as degenerate cases of Figure \ref{fig:freemedian} in which two of the dimensions of the cube are collapsed, making clear the relationship between Gromov's $4$-point condition and ours, and the sense in which coarse median spaces are a higher dimensional analogue of $\delta$-hyperbolic spaces. The presence of the central $3$-cube, which may be arbitrarily large, allows for flat geometry.

The coarse $4$-point condition is also a coarse analogue of the $4$-point characterisation for median algebras, as introduced and studied by Kolibiar and Marcisov{\'a} in \cite{kolibiar1974question}. To establish the equivalence with Bowditch's original definition we introduce a model for the free median algebra generated by $n$ points, which may be of independent interest.

Any median algebra generated by $4$ points can be modelled by the $3$-dimensional CAT(0) cube complex illustrated in Figure \ref{fig:freemedian}, so it is {\it a priori} difficult to see how to characterise rank using our new definition. We overcome this by offering several intrinsic characterisations of rank in terms of the coarse median operator itself and which are equivalent to Bowditch's definition.

The rank 1 geodesic case is of independent interest, since, as remarked above, it coincides with the class of geodesic hyperbolic spaces, \cite{bowditch2013coarse}. Bowditch's proof that rank $1$ geodesic coarse median spaces are hyperbolic uses an ingenious asymptotic cones argument which conceals in part the strong interaction between quasi-geodesics and coarse median intervals in this case. Introducing a new definition of interval in a coarse median space, we give an alternative proof, bypassing the asymptotic cones argument, and instead exploiting a result of Papasoglu \cite{papasoglu1995strongly} and Pomroy \cite{pomroy}, see also \cite{chatterji2006characterization}. We also consider the behaviour of quasi-geodesics in higher rank coarse median spaces, giving an example in rank $2$ to show that even geodesics can wander far from the intervals defined by the coarse median operator. In a subsequent paper, \cite{nwz2} we  further develop the concept of the coarse interval structure associated to a coarse median space and, as an application, we show there that the metric data is determined by the coarse median operator itself. 

The paper is organised as follows. In Section \ref{preliminaries}, we recall Bowditch's definition of coarse median spaces and introduce our notion of (coarse) intervals (which differ in one small but crucial respect from the intervals studied by Bowditch). In Section \ref{4pointcondition}, we establish our $4$-point condition characterising coarse median spaces. In Section \ref{rank}, we give several characterisations for rank in terms of the coarse median operator and give a new proof of Bowditch's result concerning the hyperbolicity of rank $1$ geodesic coarse median spaces. Finally in Section \ref{counterexample}, we construct an example to show that geodesics do not have to remain close to intervals in a coarse median space of rank greater than $1$.

\section{Preliminaries}\label{preliminaries}

\subsection{Metrics and geodesics}

\begin{defn} Let $(X,d)$ and $(Y,d')$ be metric spaces.
\begin{enumerate}
  \item $(X,d)$ is said to be \emph{quasi-geodesic}, if there exist constants $L,C>0$ such that for any two points $x,y\in X$, there exists a map $\gamma \colon [0,d(x,y)] \rightarrow X$ with $\gamma(0)=x$, $\gamma(d(x,y))=y$, satisfying: for any $s,t\in [0,d(x,y)]$,
       $$L^{-1}|s-t|-C \leqslant d(\gamma(s),\gamma(t)) \leqslant L|s-t|+C.$$
  If we care about the constants we say that $(X,d)$ is $(L,C)$-quasi-geodesic, and if we do not care about the constant $C$ we say that $(X,d)$ is \emph{$L$-quasi-geodesic}. If $(X,d)$ is $(1,0)$-quasi-geodesic then we say that $X$ is \emph{geodesic}. When considering integer-valued metrics we make the same definitions restricting the intervals to intervals in $\mathbb Z$.
  \item A map $f:(X,d)\rightarrow (Y,d')$ is \emph{bornologous} if there exists an increasing map $\rho_+:\R^+\rightarrow \R^+$ such that for all $x,y\in X$, $d'(f(x), f(y)) \leqslant \rho_+(d(x,y))$.
  \item $(X,d)$ is said to be \emph{uniformly discrete} if there exists a constant $C>0$ such that for any $x \neq y \in X$, $d(x,y)\geqslant C$.
\item Two points $x,y\in X$ are said to be \emph{$s$-close} (with respect to the metric $d$) if $d(x,y)\leqslant s$. If $x$ is $s$-close to $y$ we write $x\thicksim_s y$.
\end{enumerate}
\end{defn}

\subsection{CAT(0) Cube Complexes}

Before considering \emph{coarse} median spaces, we first recall basic notions and results about CAT(0) cube complexes. We will survey the properties we need here, but guide the interested reader to \cite{BH99, chepoi2000graphs, gromov1987hyperbolic, niblo1998geometry, sageev1995ends} for more information.

A \emph{cube complex} is a polyhedral complex in which each cell is isometric to a unit Euclidean cube and the gluing maps are isometries. The \emph{dimension} of the complex is the maximum of the dimensions of the cubes. For a cube complex $X$, we can associate it with the \emph{intrinsic pseudo-metric} $d_{int}$, which is the maximal pseudo-metric on $X$ such that each cube embeds isometrically. When $X$ is connected and has finite dimension, $d_{int}$ is a complete geodesic metric on $X$. See \cite{BH99} for a general discussion on polyhedral complex and the associated intrinsic metric. A geodesic metric space  is \emph{CAT(0)} if all its geodesic triangles are slimmer than the comparative triangle in the Euclidean space. For a cube complex $(X,d_{int})$, Gromov gave a combinatorial characterisation of the CAT(0) condition \cite{gromov1987hyperbolic}: $X$ is CAT(0) if and only if it is simply connected and the link of each vertex is a flag complex (see also \cite{BH99}).

We also consider the edge path metric $d$ on the vertex set $V$ of a CAT(0) cube complex. For $x,y\in V$, the \emph{interval} is defined to be $[x,y]=\{z\in V:d(x,y)=d(x,z)+d(x,y)\}$, which consists of points on any edge path geodesic between $x$ and $y$. A CAT(0) cubical complex $X$ can be equipped with a set of \emph{hyperplanes}~\cite{chatterji2005wall, niblo1998geometry, nica2004cubulating, sageev1995ends} such that each edge is crossed by exactly one hyperplane.  Each hyperplane divides the space into two halfspaces, and the metric $d$ counts the number of hyperplanes separating a pair of points. The dimension of $X$, if it is finite, is the maximal number of pairwise intersecting hyperplanes. We say that a subset is \emph{convex} if it is an intersection of half spaces, and we can equivalently define the interval $[x,y]$ to be the intersection of all the halfspaces containing both $x$ and $y$.

Another characterisation of the CAT(0) condition was obtained by Chepoi \cite{chepoi2000graphs} (see also \cite{roller1998poc}): a flag cube complex $X$ is CAT(0) if and only if for any $x,y,z\in V$, the intersection $[x,y] \cap [y,z] \cap [z,x]$ consists of a single point $\mu(x,y,z)$, which is called the \emph{median} of $x,y,z$. Obviously, $m(x,y,z)\in [x,y]$, and \[
[x,y]=\{m(x,y,z): z\in V\}=\{w\in V: m(x,y,w)=w\}
\]
A graph such as $X^{(1)}$ satisfying this condition is called a \emph{median graph}.

Given a CAT(0) cube complex, we always take the canonical median structure $(V,m)$ defined by intersection of intervals as above.  The pair $(V,m)$ is a \emph{median algebra} \cite{isbell1980median}, as defined in the following section.

\subsection{Median Algebras}\label{medalg}
As discussed in \cite{bandelt1983median}, there are a number of equivalent formulations of the axioms for median algebras. We will use the following formulation from \cite{kolibiar1974question}, see also \cite{bandelt2008metric}:

\begin{defn}
Let $X$ be a set and $m$ a ternary operation on $X$. Then $m$ is a \emph{median operator} and the pair $(X,m)$ is a \emph{median algebra} if:
\begin{itemize}
  \item[(M1) Localisation:] $m(a,a,b)=a$;
  \item[(M2) Symmetry:] $m(a_1,a_2,a_3)=m(a_{\sigma(1)},a_{\sigma(2)},a_{\sigma(3)})$, where $\sigma$ is any permutation of $\{1,2,3\}$;
  \item[(M3) The 4-point condition:] $m(m(a,b,c),b,d)=m(a,b, m(c,b,d))$.
\end{itemize}
\end{defn}

Condition (M3) is illustrated by Figure \ref{fig:freemedian}, which shows the free median algebra generated by the $4$ points $a,b,c,d$. As shown in the figure, the iterated medians on both sides of the equality in (M3) evaluate to the vertex adjacent to $b$.

\begin{ex}\label{mediancube}
An important example is furnished by the ``\emph{median $n$-cube}'', denoted by $I^n$, which is the $n$-dimensional vector space over $\mathbb Z_2$ with the median operator $\mu_n$ given by majority vote on each coordinate.
\end{ex}

\begin{defn}
The \emph{rank} of a median algebra $(X,m)$ is the supremum of those $n$ for which there is a subalgebra of $(X,m)$ isomorphic to the median algebra $(I^n, \mu_n)$.
\end{defn}
For the median algebra defined by the vertex set of a CAT(0) cube complex, the rank coincides with the dimension of the cube complex. 
\medskip

For any points $a,b\in X$ we \emph{define} the \emph{interval between $a,b$} to be
\[
[a,b]:=\{m(a,x,b): x\in X\}.
\]
Axioms (M1) $\sim$ (M3) ensure that $[a,b]$ is also equal to the set $\{c\in X: m(a,c,b)=c\}$, since if $c=m(a,b,x)$ then:
\[m(c,a,b)=m(m(x,a,b),a,b)=m(x,m(a,b,a),b)=m(x,a,b)=c.
\]
We think of $m(a,x,b)$ as the projection of $x$ onto the interval $[a,b]$, and axiom (M3) can be viewed as an associativity axiom: For each $b\in X$ the binary operator
\[(a,c)\mapsto a*_b c:=m(a,b,c)\]
is associative. It is also commutative by (M2) and iterated projection gives rise to the iterated median introduced in \cite{vspakula2017coarse}.

\begin{defn}[\cite{vspakula2017coarse}]
Let $(X,m)$ be a median algebra. For $x_1\in X$, define
$$m(x_1;b):=x_1,$$
and for $k \geqslant 1$ and $x_1,\ldots,x_{k+1} \in X$, define
$$m(x_1,\ldots,x_{k+1};b):=m(m(x_1,\ldots,x_k;b),x_{k+1},b).$$
Note that this definition ``agrees'' with the original median operator $m$, since $m(x_1,x_2;b)=m(x_1,x_2,b)$.
\end{defn}

In the notation above, the definition reduces to:
\[m(x_1,\ldots,x_{k};b)=x_1*_bx_2*_b \ldots *_b x_k.
\]
A subset $Y$ of $A$ is said to be \emph{convex} if $m(x,b,y)\in Y$ for all $x,y\in Y$ and $b\in X$, or equivalently, if it is closed under the binary operation $*_b$ for all $b\in X$. The set $\{~m(x_1,\ldots,x_n;b)\mid b\in X~\}$ is the convex hull of the points $x_i$ and we think of  the iterated median $m(x_1,\ldots,x_n;b)$ as the projection of $b$ onto the convex hull.

We recall several properties of the iterated median operator proved in the original paper \cite{vspakula2017coarse}.
\begin{lem}[\cite{vspakula2017coarse}]\label{itrt median}
Let $(X,m)$ be a median algebra, and $x_1,\ldots,x_n,a,b\in X$. Then:
\begin{enumerate}
  \item The iterated median operator defined above is symmetric in $x_1,\ldots,x_n$;
  \item $\bigcap\limits_{k=1}^n [x_k,b]=[m(x_1,\ldots,x_n;b),b]$;
  \item If, in addition, $X$ has rank at most $d$, then there exists a subset $\{y_1,\ldots,y_k\} \subseteq \{x_1,\ldots,x_n\}$ with $k \leqslant d$, such that
  $$m(y_1,\ldots,y_k;b)=m(x_1,\ldots,x_n;b);$$
  \item Assume that $x_1,\ldots,x_n\in[a,b]$, then $\{x_1,\ldots,x_n\} \subseteq [a,m(x_1,\ldots,x_n;b)]$.
\end{enumerate}
\end{lem}

We remark that condition (1) here follows immediately from the commutativity and associativity of the binary operator $*_b$.

We will also make use of the following ``(n+2)-point condition''.
\begin{lem}\label{itrt eqn}
Let $(X,m)$ be a median algebra and $a,b,e_1,\ldots,e_{n+1}\in X$, then
$$m(a,e_{n+1},m(e_1,\ldots,e_n;b))=m(m(a,e_{n+1},e_1),\ldots,m(a,e_{n+1},e_n);b).$$
\end{lem}

\begin{proof}
We prove this by induction on $n$.

$n=1$: the equation holds trivially. Assume it holds for all $k\leqslant n$. Then
\begin{eqnarray*}
&&m(a,e_{n+1},m(e_1,\ldots,e_n;b)) = m\big(a,e_{n+1},m(m(e_1,\ldots,e_{n-1};b),e_n,b)\big)\\
                                 &=& m\big(m(a,e_{n+1},m(e_1,\ldots,e_{n-1};b)),m(a,e_{n+1},e_n),b\big)\\
                                 &=& m\big(m(m(a,e_{n+1},e_1),\ldots,m(a,e_{n+1},e_{n-1});b),m(a,e_{n+1},e_n),b\big)\\
                                 &=& m(m(a,e_{n+1},e_1),\ldots,m(a,e_{n+1},e_n);b),
\end{eqnarray*}
where we use the inductive assumption in the third equation.
\end{proof}

We now mention two alternative definitions for median algebras.

According to Isbell \cite{isbell1980median}, a ternary operator $m$ defines a median \emph{if and only if} it satisfies (M1), (M2) and \emph{Isbell's condition}:
\[
m(a,m(a, b,c),m(b,c,d))=m(a,b,c).
\]
This says that $m(a,b,c)$ is in the interval from $a$ to $m(b,c,d)$, or in geometrical terms that the projection of $a$ onto $[b,c]$ provided by the median lies between $a$ and every other point ($m(b,c,d)$) of the interval $[b,c]$.

An alternative and algebraically powerful formulation, is that $(X,m)$ is a median algebra if it satisfies (M1), (M2) and the \emph{five point condition}:
\[
m(m(a,b,c),d,e)=m(a,m(b,d,e),m(c,d,e)),
\]
which is the $n=3$ case of Lemma \ref{itrt eqn}.
Setting $e=b$ one recovers (M3), while
\[
m(a,m(d,b,c),(a,b,c))=m(m(a,d,a),b,c)=m(a,b,c)
\]
recovering Isbell's condition.

\subsection{Coarse median spaces}\label{cma}

In \cite{bowditch2013coarse}, Bowditch introduced coarse median operators as follows:
\begin{defn}[Bowditch, \cite{bowditch2013coarse}]
Given a metric space $(X,d)$, a \emph{coarse median (operator) on $X$} is a ternary operator $\mu\colon X^3 \rightarrow X$ satisfying the following conditions:
\begin{itemize}
  \item[(C1).] There is an affine function $\rho(t)=Kt+H_0$ such that for any $a,b,c,a',b',c' \in X$,
        $$ d(\mu(a,b,c), \mu(a',b',c')) \leqslant \rho(d(a,a')+d(b,b')+d(c,c')).$$
  \item[(C2).] There is a function $H\colon \mathbb{N} \rightarrow [0,+\infty)$, such that for any finite subset $A\subseteq X$ with $1 \leqslant |A| \leqslant p$, there exists a finite median algebra $(\Pi, \mu_{\Pi})$ and maps $\pi\colon A \rightarrow \Pi$, $\lambda\colon  \Pi \rightarrow X$ such that for any $x,y,z \in \Pi, a\in A$,
      $$\lambda \mu_{\Pi}(x,y,z) \thicksim_{H(p)} \mu(\lambda x, \lambda y, \lambda z),$$
      and
      $$\lambda \pi a \thicksim_{H(p)} a.$$
      We may assume that $\Pi$ is generated by $\pi(A)$.
\end{itemize}
We say that two coarse median operators $\mu_1, \mu_2$ on a metric space $(X,d)$ are  \emph{uniformly close} if there is a uniform bound on the set of distances
\[\{d(\mu_1(a,b,c), \mu_2(a,b,c))\mid a,b,c\in X\}.
\]
\end{defn}

\begin{rem}
The control function $\rho$ in (C1) is required by Bowditch to be affine, however it seems more natural in the context of coarse geometry to allow $\rho$ to be an \emph{arbitrary } function and indeed much of what we show in this paper works with that variation. This may allow wider applications in the future (in \cite{nwz2} we introduce and study such a generalisation in the context of coarse interval structures), though we note that when $X$ is a quasi-geodesic space the existence of any  control function $\rho$ guarantees that $\rho$ may be replaced by an affine control function as in (C1). In this paper we will assume that $\rho$ is affine throughout, but note that many of the statements and arguments can be suitably adapted to the more general context.
\end{rem}

We refer to the functions $\rho,H$ appearing in conditions (C1), (C2) as \emph{parameters} of the coarse median, or since $\rho$ has the form $\rho(t)=Kt+H_0$ we will sometimes refer to $K,H_0,H$ as parameters of the coarse median. The parameters are not unique, and neither are they part of the data of the coarse median; it is merely their existence which is required.

\begin{rem}\label{median assump}
As noted by Bowditch there is a constant $\kappao>0$ such that:
\begin{itemize}
  \item $\mu(a,a,b)\thicksim_{\kappao} a$;
  \item $\mu(a_1,a_2,a_3)\thicksim_{\kappao} \mu(a_{\sigma(1)},a_{\sigma(2)},a_{\sigma(3)})$ for any permutation $\sigma\in S_3$.
\end{itemize}
It follows that any coarse median operator on $(X,d)$ may be replaced by another to which it is uniformly close and which satisfies the first two median axioms (M1) and (M2), i.e. we may always assume that $\kappao=0$.
We note here that this is true even if we replace the affine control function by an arbitrary $\rho$ as discussed above, taking $\kappao=2\rho(3H(3))+2H(3)$.
\end{rem}

With this in mind we take the following as our definition of a coarse median space.
\begin{defn}\label{def for cma}
A \emph{coarse median space} is a triple $(X,d,\mu)$ where $d$ is a metric on $X$ and $\mu$ is a ternary operator on $X$ satisfying conditions (M1), (M2), (C1) and (C2).

A map $f$ between coarse median spaces $(X,d_X,\mu_X), (Y,d_Y,\mu_Y)$ is said to be an \emph{$L$-quasi-morphism} if for any $a,b,c\in X$, $\mu_Y(f(a),f(b),f(c))\thicksim_L f(\mu_X(a,b,c))$.
\end{defn}

We note that given a median algebra any metric on this will satisfy axiom (C2), however the metric must be chosen carefully if we wish it to also satisfy axiom (C1). Of course in the case that the median algebra has finite intervals, and hence arises as the vertex set of a CAT(0) cube complex, then both the intrinsic and edge path metrics satisfy this axiom. 

\subsection{The rank of a coarse median space}
Rank is a proxy for dimension in the context of coarse median spaces, directly analogous to the notion of dimension for a  CAT(0) cube complex.

\begin{defn}\label{rankdef}
Let $n$ be a natural number. We say $X$ has \emph{rank at most $n$} if there exist parameters $\rho,H$ for which we can always choose the approximating median algebra $\Pi$ in condition (C2) to have rank at most $n$.
\end{defn}

We remark that for a median algebra equipped with a suitable metric making it a coarse median space, the rank as a median algebra gives an upper bound for the rank as a coarse median space, however these need not agree. For example any finite median algebra has rank $0$ as a coarse median space. 

Zeidler \cite{zeidler2013coarse} showed that, at the cost of increasing the rank of the approximating median algebra, one can always assume that the map $\lambda \pi$ from condition (C2) is the inclusion map $\iota_A \colon A \hookrightarrow X$. We will now  show that this can be achieved without increasing the rank:

\begin{lem}\label{median rank assump}
Assume $(X,d,\mu)$ is a coarse median space. Then, at the cost of changing the parameter function $H$, one can always change the triples $\Pi, \lambda, \pi$ provided by axiom (C2), so that $\lambda \pi=\iota_A \colon A \hookrightarrow X$ (the inclusion map), without changing the rank of $\Pi$.
\end{lem}

\begin{proof}
Given a finite subset $A \subseteq X$ with $1 \leqslant |A| \leqslant p$, let the finite median algebra $(\Pi, \mu_{\Pi})$ and maps $\pi\colon A \rightarrow \Pi$, $\lambda \colon \Pi \rightarrow X$ be as in the definition above.

We now construct another finite median algebra $(\Pi', \mu'_\Pi)$. As a set, $\Pi'=\Pi \sqcup A$. Define a map $\tau \colon \Pi' \rightarrow \Pi$ by $\tau x=x$ if $x\in \Pi$, and $\tau a=\pi a$ if $a\in A$. Now define a median $\mu'_\Pi$ on $\Pi'$ by:
\begin{itemize}
  \item $\mu'_\Pi(a,a,x)=\mu'_\Pi(a,x,a)=\mu'_\Pi(x,a,a)=a$, if $a\in A$ and $x\in \Pi'$;
  \item $\mu'_\Pi(x,y,z)=\mu_\Pi(\tau x,\tau y,\tau z)$, otherwise.
\end{itemize}
It is a direct calculation to check that $\mu_\Pi'$ satisfies the axioms of a median operator.
Now define $\pi' \colon A \rightarrow \Pi'$ by $\pi'a=a$; and $\lambda' \colon \Pi' = \Pi \sqcup A \rightarrow X$ by $\lambda'=\lambda \sqcup \iota_A$.
For any $x,y,z\in \Pi'$: if two of them are equal and sit in $A$, say $x=y\in A$, then
$$\lambda'\mu_\Pi'(x,y,z)=\lambda'x=x=\mu(\lambda'x, \lambda'y, \lambda'z);$$
otherwise, we have:
\begin{eqnarray*}
\lambda'\mu_\Pi'(x,y,z) & = &\lambda'\mu_\Pi(\tau x,\tau y,\tau z) \quad=\quad \lambda\mu_\Pi(\tau x,\tau y,\tau z)\\
                        & \thicksim_{H(p)} & \mu(\lambda'\tau x ,\lambda'\tau y ,\lambda'\tau z) ~~\thicksim_{\rho(3H(p))}~~ \mu(\lambda'x, \lambda'y, \lambda'z),
\end{eqnarray*}
where in the last estimate we use (C1) and the fact that for any $x\in A$,
$$\lambda'\tau x = \lambda' \pi x= \lambda \pi x \thicksim_{H(p)} x=\lambda'x.$$
Now for any $a\in A$, $\lambda'\pi'a=\lambda'a=a$, and by the construction, it is obvious that $\rank{\Pi'}=\rank{\Pi}$.

In conclusion we have constructed $\Pi',\lambda',\pi'$ such that $\lambda'\pi'$ is the inclusion, $\Pi'$ has the same dimension as $\Pi$ and
$$\lambda' \mu'_{\Pi}(x,y,z) \thicksim_{H'(p)} \mu(\lambda' x, \lambda' y, \lambda' z),$$
where  $H' \colon p \mapsto \rho(3H(p))+H(p)$.
\end{proof}

According to the above lemma there is no loss of generality in assuming that the triples $\Pi, \lambda, \pi$ provided by axiom (C2) satisfy the additional condition that $\lambda\pi$ is the inclusion map. Hereafter we will assume that parameters $\rho,H$ (or $K,H_0,H$) for a coarse median are chosen such that this holds.

In the finite rank case we introduce the following terminology.
\begin{defn}\label{convention}
For a coarse median space $(X,d,\mu)$, we say that $\rank X \leqslant n$ can be \emph{achieved under parameters $H,\rho$} (or $K,H_0,H$) if one can always choose $\Pi$ in condition (C2) with rank at most $n$ and $\lambda \pi=\iota_A \colon A \hookrightarrow X$. By Lemma \ref{median rank assump} there is no change to the definition of rank.
\end{defn}

\subsection{Iterated coarse medians}\label{iteratedcoarsemedians}

By analogy with the iterated medians defined in Section \ref{medalg}, we can define the \emph{iterated coarse median operator} in a coarse median space.

\begin{defn}\label{coarseiteratedmediandefn}
Let $(X,d,\mu)$ be a coarse median space and $b\in X$. For $x_1\in X$ define
$$\mu(x_1;b):=x_1,$$
and for $k \geqslant 1$ and $x_1,\ldots,x_{k+1} \in X$, define
$$\mu(x_1,\ldots,x_{k+1};b):=\mu(\mu(x_1,\ldots,x_k;b),x_{k+1},b).$$
Note that this definition ``agrees'' with the original coarse median operator $\mu$ in the sense that for any $a,b,c$ in $X$, $\mu(a,b,c)=\mu(a,b;c)$.
\end{defn}

We extend the estimates provided by axioms (C1) and (C2) for a coarse median operator to hold more generally for the iterated coarse median operators as follows:

\begin{lem}\label{itrt C1}
Let $(X,d)$ be a metric space with ternary operator $\mu$ satisfying (C1) with parameter $\rho$. Then for any $n$ there exists an increasing (affine)  function $\rho_n$ depending on $\rho$, such that for any $a_0, a_1,\ldots,a_n,b_0,b_1,\ldots,b_n \in X$:
\[
d(\mu(a_1,\ldots,a_n;a_0),\mu(b_1,\ldots,b_n;b_0)) \leqslant \rho_n \big(\sum_{k=0}^n d(a_k,b_k)\big).
\]
\end{lem}

\begin{proof}
We carry out induction on $n$. This is trivial for $n=1$ with $\rho_1(t)=t$. Now consider the case $n>1$ and assume that the result holds for $n-1$:
\begin{eqnarray*}
&&d(\mu(a_1,\ldots,a_n;a_0),\mu(b_1,\ldots,b_n;b_0))\\
 &=& d\big(\mu(\mu(a_1,\ldots,a_{n-1};a),a_n,a_0),\mu(\mu(b_1,\ldots,b_{n-1};b),b_n,b_0)\big)\\
 &\leqslant & \rho\big(d(\mu(a_1,\ldots,a_{n-1};a_0),\mu(b_1,\ldots,b_{n-1};b_0))+d(a_n,b_n)+d(a_0,b_0)\big)\\
 &\leqslant & \rho \big(\rho_{n-1} \big(\sum_{k=0}^{n-1} d(a_k,b_k)\big) + d(a_n,b_n)+d(a_0,b_0)\big)\\
 &\leqslant & \rho \big(\rho_{n-1} \big(\sum_{k=0}^{n} d(a_k,b_k)\big) + \sum_{k=0}^{n} d(a_k,b_k)\big)\\
 &= & \rho_n\big(\sum_{k=0}^n d(a_k,b_k)\big),
\end{eqnarray*}
where $\rho_n(t):=\rho(\rho_{n-1}(t)+t)$. We use (C1) in the third line, and the inductive assumption in the fourth. Note that  as $\rho_{n-1}$ is increasing $\rho_n$ is also increasing.
\end{proof}

Recall that in Bowditch's definition of a coarse median space axiom (C2) says that for any finite $A \subseteq X$ with $|A|\leqslant p$, the approximation map $\sigma: (\Pi,\mu_\Pi) \rightarrow (X,\mu)$ is an $H(p)$-quasi-morphism, i.e. for any $x,y,z\in \Pi$,
$$\sigma \mu_{\Pi}(x,y,z) \thicksim_{H(p)} \mu(\sigma x, \sigma y, \sigma z).$$

\begin{lem}\label{itrt C2}
Let $(X,d)$ be a metric space with ternary operator $\mu$ satisfying (C1) with parameter $\rho$, $(\Pi,\mu_\Pi)$ a median algebra, and $\sigma \colon \Pi \rightarrow X$ an $L$-quasi-morphism. Then there exists a constant $H_n(L)$ (depending on $\rho$) such that for any $x_1,\ldots,x_n,b\in\Pi$,
$$\sigma(\mu_\Pi(x_1,\ldots,x_n;b))\thicksim_{H_n(L)}\mu(\sigma(x_1),\ldots,\sigma(x_n);\sigma(b)).$$
\end{lem}

\begin{proof}
We carry out induction on $n$. For the case $n=1$, set $H_1(L)=0$; and for the case $n=2$, set $H_2(L)=L$. Now assume $n>2$, for any $x_1,\ldots,x_n,b\in\Pi$, we have:
$$\sigma(\mu_\Pi(x_1,\ldots,x_{n-1};b))\thicksim_{H_{n-1}(L)}\mu(\sigma(x_1),\ldots,\sigma(x_{n-1});\sigma(b)).$$
Then
\begin{eqnarray*}
\sigma(\mu_\Pi(x_1,\ldots,x_n;b))&=& \sigma(\mu_\Pi(\mu_\Pi(x_1,\ldots,x_{n-1};b),x_n,b)) \\
&\thicksim_L& \mu(\sigma(\mu_\Pi(x_1,\ldots,x_{n-1};b)),\sigma(x_n),\sigma(b))\\
&\thicksim_{\rho(H_{n-1}(L))}& \mu(\mu(\sigma(x_1),\ldots,\sigma(x_{n-1});\sigma(b)),\sigma(x_n),\sigma(b))\\
&=& \mu(\sigma(x_1),\ldots,\sigma(x_n);\sigma(b)),
\end{eqnarray*}
where we use the definition of quasi-morphism in the second line, and (C1) as well as the inductive assumption in the third line. Finally, take $H_n(L)=\rho(H_{n-1}(L))+L$ for $n>2$.
\end{proof}

Now we prove a coarse version of Lemma \ref{itrt eqn}. Note that the case $n=1$ is precisely the coarse analogue of the five point condition for a median algebra.
\begin{lem}\label{itrt eqn coarse}
Let $(X,d,\mu)$ be a coarse median space with parameters $\rho,H$, then there exists a constant $C_n$ depending on $\rho,H$ such that for any $a,b,e_1,\ldots,e_{n+1}\in X$,
$$\mu(a,e_{n+1},\mu(e_1,\ldots,e_n;b))\thicksim_{C_n}\mu(\mu(a,e_{n+1},e_1),\ldots,\mu(a,e_{n+1},e_n);b).$$
\end{lem}

\begin{proof}
We prove this by induction on $n$.

The case $n=1$ reduces to establishing a coarse analogue of the median algebra five point condition which was established in \cite{bowditch2013coarse}, specifically there is a constant
$$\kappav = \rho(H(5))+\rho(2H(5))+2H(5)$$
such that:
$\forall x,y,z,v,w \in X$,
\begin{equation}\label{main est}
\mu(x,y,\mu(z,v,w)) \thicksim_\kappav \mu(\mu(x,y,z),\mu(x,y,v),w).
\end{equation}
So we can take $C_1=\kappav$.

Now assume that $n>1$ and the result holds for $n-1$. Then
\begin{eqnarray*}
&&\mu(a,e_{n+1},\mu(e_1,\ldots,e_n;b)) = \mu\big(a,e_{n+1},\mu(\mu(e_1,\ldots,e_{n-1};b),e_n,b)\big)\\
&\thicksim_{\kappav}& \mu\big(\mu(a,e_{n+1},\mu(e_1,\ldots,e_{n-1};b)),\mu(a,e_{n+1},e_n),b\big)\\
&\thicksim_{\rho(C_{n-1})}& \mu\big(\mu(\mu(a,e_{n+1},e_1),\ldots,\mu(a,e_{n+1},e_{n-1});b),\mu(a,e_{n+1},e_n),b\big)\\
&=& \mu(\mu(a,e_{n+1},e_1),\ldots,\mu(a,e_{n+1},e_n);b),
\end{eqnarray*}
where we use the inductive assumption in the third line. Set $C_n=\rho(C_{n-1})+\kappav$ and we are done.
\end{proof}

We note that this result still holds if we replace the affine control function $\rho$ by  an arbitrary increasing control function.

We conclude our consideration of iterated coarse medians with the following lemma.

\begin{lem}\label{itrt eqn coarse2}
Let $(X,d,\mu)$ be a coarse median space with parameters $\rho,H$, then there exists a constant $D_n$ depending on $\rho,H$ such that for any $a,b,c,e_1,\ldots,e_n\in X$,
$$\mu(a,b,\mu(e_1,\ldots,e_n;c))\thicksim_{D_n}\mu(\mu(a,b,e_1),\ldots,\mu(a,b,e_n);\mu(a,b,c)).$$
\end{lem}

\begin{proof}
The proof is by induction; the case $n=1$ is elementary setting $D_1=0$. Now assume $n>1$ and the lemma holds for $n-1$. Then
\begin{eqnarray*}
&&\mu(a,b,\mu(e_1,\ldots,e_n;c)) = \mu\big(a,b,\mu(\mu(e_1,\ldots,e_{n-1};c),e_n,c)\big)\\
&\thicksim_\kappav& \mu\left( a,b,\mu\big(a,b,\mu(\mu(e_1,\ldots,e_{n-1};c),e_n,c)\big) \right)\\
&\thicksim_{\rho(\kappav)}& \mu\left( a,b, \mu\big(\mu(a,b,\mu(e_1,\ldots,e_{n-1};c),\mu(a,b,e_n),c)\big)\right)\\
&\thicksim_{\kappav}& \mu\left(\mu(a,b,\mu(e_1,\ldots,e_{n-1};c), \mu(a,b,\mu(a,b,e_n)),\mu(a,b,c)\big)\right)\\
&\thicksim_{\rho(\kappav)}& \mu\left(\mu(a,b,\mu(e_1,\ldots,e_{n-1};c), \mu(a,b,e_n),\mu(a,b,c)\big)\right)\\
&\thicksim_{\rho(D_{n-1})}& \mu\left(\mu(\mu(a,b,e_1),\ldots,\mu(a,b,e_{n-1});\mu(a,b,c)), \mu(a,b,e_n),\mu(a,b,c)\big)\right)\\
&=&\mu(\mu(a,b,e_1),\ldots,\mu(a,b,e_n);\mu(a,b,c))
\end{eqnarray*}
where we use the inductive assumption in the sixth inequality. Set $D_n=\rho(D_{n-1})+2\rho(\kappav)+2\kappav$ and we are done.
\end{proof}

Again this result holds in the context of arbitrary control functions.

\subsection{(Coarse) intervals}\label{coarseintervals}
In CAT(0) cube complexes intervals play an important role, indeed the natural median is determined by the interval structure and vice versa. Similarly, in coarse median spaces, one needs to consider  coarse analogues of intervals. Some of these were introduced by Bowditch \cite{bowditch2014embedding}.

\begin{defn}
Given $(X,\mu)$ a set equipped with a ternary operator $\mu$, we define the \emph{interval} between points $x$ and $z$ to be:
$$[x,z]=\{\mu(x,y,z): y\in X\}.$$
\end{defn}

This should be contrasted with Bowditch's definition, \cite{bowditch2014embedding}, of the $\lambda$-\emph{coarse interval} between points $x$ and $z$ in a coarse median space $(X,d,\mu)$ as:
$$[x,z]_\lambda=\{y\in X: \mu(x,y,z)\thicksim_\lambda y\}.$$
Clearly $[x,z]_0\subseteq [x,z]$ and, as noted in Section \ref{medalg}, if $(X,\mu)$ is a median algebra then $[x,z]_0=[x,z]$, however these two notions of interval do not always coincide in a coarse median space. An example is provided in Section \ref{counterexample}.

In a CAT(0) cube complex the median of three points is always the unique point in the intersection of the three intervals they define. Bowditch \cite{bowditch2014embedding} showed the same result holds coarsely in a coarse median space. We adapt this to our notion of interval as follows (as usual the result works whether or not the control function is affine).

\begin{lem}\label{kappav}
Let $(X,d,\mu)$ be a coarse median space with parameters $\rho,H$. Then for any $x,y,z\in X$, $\mu(x,y,z)\in [x,y]_\kappavv$, i.e. $[x,y] \subseteq [x,y]_\kappavv$, where $\kappav$ is the constant $\rho(H(5))+\rho(2H(5))+2H(5)$ defined above.  In particular,
$$\mu(x,y,z) \in [x,y]_\kappavv \cap [y,z]_\kappavv \cap [x,z]_\kappavv.$$
\end{lem}

Note that generally, the coarse interval $[a,b]_\kappavv$ is not closed under the coarse median $\mu$. However, we can endow it with another coarse median $\mu_{a,b}$ such that $[a,b]_\kappavv$ is closed under $\mu_{a,b}$, and $\mu_{a,b}$ is uniformly closed to $\mu$ as follows:
\begin{lem}\label{redefine cma on intvl}
Let $(X,d,\mu)$ be a coarse median space, and $\kappav$ be as defined above. Then there exists a constant $C>0$ depending  on the chosen parameters $\rho,H$ (and not on $(X,d,\mu)$ itself), such that for any $a,b \in X$ and $x,y,z\in [a,b]_\kappavv$, $\mu_{a,b}(x,y,z) := \mu(a,b,\mu(x,y,z))$ is $C$-close to $\mu(x,y,z)$. In conclusion, $\mu_{a,b}$ is indeed a coarse median on the coarse interval $[a,b]_\kappavv$, and $\mu_{a,b}$ is uniformly close to $\mu$. The same result holds if we use $[a,b]$ instead of $[a,b]_{\kappav}$.
\end{lem}

\begin{proof}
In the above setting, we have
$$\mu_{a,b}(x,y,z) = \mu(a,b,\mu(x,y,z)) \thicksim_\kappav\mu(\mu(a,b,x),\mu(a,b,y),z) \thicksim_{\rho(2\kappav)} \mu(x,y,z),$$
where in the last estimate we use that fact that $x,y\in [a,b]_{\kappav}$.

Finally, $\mu_{a,b}(x,y,z) = \mu(a,b,\mu(x,y,z))$ sits in $[a,b] \subseteq [a,b]_{\kappav}$.
\end{proof}

\section{The 4-point condition defines a coarse median space}\label{4pointcondition}
A fundamental difficulty with verifying that a space satisfies Bowditch's axioms is that one needs to establish approximations for subsets of arbitrary cardinality.  Here we will provide an alternative characterisation of coarse median spaces, for which one need only consider subsets of cardinality up to $4$.

Applying the axioms (C1), (C2), it is not hard to show that there is a constant $\kappaiv=2\rho(H(4))+2H(4)$ such that in any coarse median space $(X,d,\mu)$ with parameters $\rho,H$, for any points $a,b,c,d$ we have:
$$\mu(\mu(a,b,c),b,d) \thicksim_{\kappaiv} \mu(a,b, \mu(c,b,d)).$$
This provides a coarse analogue of Kolibiar's 4-point axiom (M3). We will provide the following converse:

\begin{thm}\label{simp.cma}
Let $(X,d)$ be a metric space, and $\mu\colon X^3 \rightarrow X$ a ternary operation. Then $\mu$ is a coarse median on $(X,d)$ (i.e., it satisfies conditions (C1) and (C2)) if and only if the following three conditions hold:
\begin{itemize}
  \item[(C0)' {\rm Coarse localisation and coarse symmetry:}] There is a constant $\kappao$ such that for all points $a_1,a_2,a_3$ in $X$, $\mu(a_1,a_1,a_2)\thicksim_{\kappao} a_1$, and $\mu(a_{\sigma(1)},a_{\sigma(2)},a_{\sigma(3)})\thicksim_{\kappao} \mu(a_1,a_2,a_3)$ for any permutation $\sigma$ of $\{1,2,3\}$;
  \item[(C1)' {\rm Affine control:}] There exists an affine function $\rho:[0,+\infty)\to [0,+\infty)$ such that for all $a,a',b,c\in X$,
      $$d(\mu(a,b,c), \mu(a',b,c)) \leqslant \rho(d(a,a'));$$
  \item[(C2)' {\rm Coarse 4-point condition:}] There exists a constant $\kappaiv>0$ such that for any $a,b,c,d\in X$, we have:
      $$\mu(\mu(a,b,c),b,d) \thicksim_{\kappaiv} \mu(a,b, \mu(c,b,d)).$$
  \end{itemize}
\end{thm}

We note that if we are interested in generalising coarse median spaces to allow arbitrary control functions then the affine control axiom (C1)' should be replaced by the requirement that the map $a\mapsto \mu(a,b,c)$ is bornologous uniformly in $b,c$.

Free median algebras will play a crucial role in the proof and next we will give a concrete construction of the free median algebra on $p$ points  as a quotient of the space of formal ternary expressions on $p$ variables. Under the name of formal median expressions these were originally introduced by Zeidler in \cite{zeidler2013coarse}.

Throughout this section, we fix an integer $p>0$ and variables $\alpha_1,\ldots,\alpha_p$.

\subsection{Formal ternary expressions and formal median identities}\label{formal}
We introduce a model to define formal ternary expressions on the alphabet $\Omega_p=\{\alpha_1,\ldots,\alpha_p\}$. We introduce two additional symbols: $"\ml"$ and $"\mr"$, and set $\widetilde{\Omega}_p=\Omega_p \cup \{~\ml~,~\mr~\}$. Let  $\widetilde{\Omega}_p^\star$ be the set of all finite words in $\widetilde{\Omega}_p$.

\begin{defn}
Define $\mathfrak{M}_p$ to be the unique smallest subset $A$ in $\widetilde{\Omega}_p^\star$ containing $\Omega_p$ and satisfying: for any $\varphi_1,\varphi_2,\varphi_3\in A$, we have $\ml\varphi_1\varphi_2\varphi_3\mr\in A$. Elements in $\mathfrak{M}_p$ are called \emph{formal ternary expressions in variables} $\alpha_1,\ldots,\alpha_p$, or \emph{formal ternary expressions with} $p$ \emph{variables}. Note that $\mathfrak{M}_p$ carries a natural ternary operation:
\[ (\varphi_1, \varphi_2, \varphi_3)\mapsto \langle \varphi_1\varphi_2\varphi_3\rangle.
\]
\end{defn}

\begin{lem}\label{basic fme}
For any $\varphi \in \mathfrak{M}_p \setminus \Omega_p$, there exist unique $\varphi_1,\varphi_2,\varphi_3\in \mathfrak{M}_p$ such that $\varphi=\ml\varphi_1\varphi_2\varphi_3\mr$.
\end{lem}

\begin{proof}
Each word in $\mathfrak{M}_p \setminus \Omega_p$ has the form $\ml\varphi_1\varphi_2\varphi_3\mr$ since otherwise we can delete the words not satisfying this condition to get a smaller set, which is a contradiction to the minimality of $\mathfrak{M}_p$.

Now we claim: for any $\varphi,\psi \in \mathfrak{M}_p$, if $\psi$ is a prefix of $\varphi$ (i.e. there exists a word $w\in \widetilde{\Omega}_p^\star$ such that $\varphi=\psi w$), then they are equal. If the claim holds then it follows that each $\varphi \in \mathfrak{M}_p \setminus \Omega_p$ can be written uniquely in the form $\ml\varphi_1\varphi_2\varphi_3\mr$. Indeed, assume $\varphi=\ml\varphi_1\varphi_2\varphi_3\mr = \ml\psi_1\psi_2\psi_3\mr$ for $\varphi_i, \psi_i \in \mathfrak{M}_p$. Since $\psi_1$ is a prefix of $\varphi_1$ or vice versa, so by the claim, we have $\psi_1=\varphi_1$; similarly, we have $\psi_2=\varphi_2$ and $\psi_3=\varphi_3$.

Now we prove the claim by induction on the word length of $\varphi$. The claim holds trivially for any $\varphi$ of word length $1$, i.e. $\varphi\in \Omega_p$. Assume that $\varphi$ has word length $n$ and that for any word of length less than $n$ the claim holds. Suppose that $\psi$ is a prefix of $\varphi$, and assume $\psi=\ml\psi_1\psi_2\psi_3\mr$ and $\varphi=\ml\varphi_1\varphi_2\varphi_3\mr$ for $\varphi_i, \psi_i \in \mathfrak{M}_p$. As $\psi$ is a prefix of $\varphi$, we know $\psi_1$ is a prefix of $\varphi_1$ or vice versa. Inductively we have $\psi_1=\varphi_1$. Similarly, we have $\psi_2=\varphi_2$ and $\psi_3=\varphi_3$, so $\psi=\varphi$, as required.
\end{proof}

\begin{defn}
For each formal ternary expression $\varphi \in \mathfrak{M}_p$, we can associate a natural number $\xi(\varphi)$ to it inductively as follows, which is called the \emph{complexity} of $\varphi$:
\begin{itemize}
  \item If $\varphi \in \Omega_p$, i.e. $\varphi=\alpha_i$ for some $i$, we define $\xi(\alpha_i)=0$;
  \item Assume we have defined complexities for all the formal ternary expressions with length (as a word in $\widetilde{\Omega}_p^\star$) less than $n>1$, then for $\varphi$ with length $n$, by Lemma \ref{basic fme}, there exists unique $\varphi_1,\varphi_2,\varphi_3 \in \mathfrak{M}_p$ such that $\varphi=\ml\varphi_1\varphi_2\varphi_3\mr$. Since each $\varphi_i$ has length less than $n$, $\xi(\varphi_i)$ has already been defined. Now we define the complexity $\xi(\varphi):= \max\{\xi(\varphi_1),\xi(\varphi_2),\xi(\varphi_3)\}+1$.
\end{itemize}
Let $\mathfrak{M}_p(n) \subseteq \mathfrak{M}_p$ be the set of all formal ternary expressions having complexity less than or equal to $n$, and $\mathfrak{S}_p(n)=\mathfrak{M}_p(n) \setminus \mathfrak{M}_p(n-1)$ be the set of all formal ternary expressions having complexity equal to $n$ (set $\mathfrak{S}_p(0)=\mathfrak{M}_p(0)$).
\end{defn}

\begin{rem}
It's obvious by definition that
$$\mathfrak{M}_p=\bigcup_{n=0}^\infty \mathfrak{M}_p(n) = \bigsqcup_{n=0}^\infty \mathfrak{S}_p(n),$$
and, by an easy inductive argument, $\mathfrak{M}_p(n)$ is finite for each $n$. Lemma \ref{basic fme} says that for any formal ternary expression $\varphi \notin \mathfrak{M}_p(0)$, there exist unique formal ternary expressions $\varphi_1,\varphi_2,\varphi_3\in \mathfrak{M}_p$ with $\xi(\varphi_i)<\xi(\varphi)$ and $\varphi=\ml\varphi_1\varphi_2\varphi_3\mr$.
\end{rem}

\begin{lem}\label{free}
Let $Y$ be a set with ternary operation $\mu$ and $\mathfrak{M}_p$ be the set of formal ternary expressions on the alphabet $\Omega_p$. Then any map $\Omega_p\rightarrow Y$ extends uniquely to a map of ternary algebras $\mathfrak{M}_p\rightarrow Y$. Thus $\mathfrak{M}_p$ is the free ternary algbera on $p$ variables.
\end{lem}

\begin{proof}
Let $y_1,\ldots,y_p \in Y$ be the images of the variables $\{\alpha_1, \ldots, \alpha_p\}$ in $\Omega_p$. Given $\varphi\in \mathfrak{M}_p$ we define an element $\varphi_Y(y_1,\ldots,y_p)$ inductively as follows:
\begin{enumerate}
  \item For $\varphi=\alpha_i\in \mathfrak{M}_p(0)$, define $\varphi_Y(y_1,\ldots,y_p)=y_i$;
  \item If $\varphi=\ml\varphi_1\varphi_2\varphi_3\mr\in \mathfrak{M}_p(n), n>0$, $\varphi_1, \varphi_2,\varphi_3$ all lie in $\mathfrak{M}_p(n-1)$ and we define: $$\varphi_Y(y_1,\ldots,y_p)=\mu\big((\varphi_1)_Y(y_1,\ldots,y_p),(\varphi_2)_Y(y_1,\ldots,y_p),(\varphi_3)_Y(y_1,\ldots,y_p)\big).$$
\end{enumerate}
By construction the map $\varphi\mapsto \varphi_Y(y_1, \ldots, y_p)$ is a map of ternary algebras extending the map from $\Omega_p$ to $Y$ as required.

Conversely, the condition that we have a map of ternary algebras extending the original map implies that $\varphi_Y$ must satisfy conditions $1$ and $2$ above yielding uniqueness.
\end{proof}

Let $Y$ be a set with ternary operation $\mu$, and $y_1,\ldots,y_p \in Y$. Given a formal ternary expression $\varphi$ on $p$ variables, we define its evaluation at $(y_1,\ldots,y_p)$ to be the image $\varphi_Y(y_1, \ldots, y_p)$ of $\varphi$ provided by Lemma \ref{free}. We call the corresponding map $\varphi_Y:Y^p\rightarrow Y$  the \emph{realisation of $\varphi$}.

Given two formal ternary expressions $\varphi ,\psi \in \mathfrak{M}_p$, we say that the equation $\varphi=\psi$ is a \emph{formal median identity}, if it holds in any median algebra. To be more precise, for any median algebra $(Y,\mu)$ and $y_1,\ldots,y_p\in Y$, we have:
\[\varphi_Y(y_1,\ldots,y_p)=\psi_Y(y_1,\ldots,y_p).
\]
It is easy to see that the axioms for a median operator immediately lead to the following formal median identities:
$$
\left.
   \begin{array}{rclcl}
      \varphi &=& \langle \varphi\varphi\psi\rangle, &\quad& \forall \varphi,\psi\in \mathfrak{M}_p;\\
      \ml\varphi_1\varphi_2\varphi_3\mr &=& \ml\varphi_{\sigma(1)}\varphi_{\sigma(2)}\varphi_{\sigma(3)}\mr, &\quad& \forall  \varphi_1,\varphi_2,\varphi_3\in \mathfrak{M}_p \text{ and } \sigma \in S_3;\\
      \ml\ml\varphi_1\varphi_2\varphi_3\mr\varphi_2\varphi_4\mr &=&\ml\varphi_1\varphi_2\ml\varphi_3\varphi_2\varphi_4\mr\mr, &\quad& \forall \varphi_1,\ldots,\varphi_4\in \mathfrak{M}_p.
   \end{array}
\right.
$$
Two formal ternary expressions are said to be \emph{equivalent}, written $\varphi \thicksim_{med} \psi$, if $\varphi=\psi$ is a formal median identity. Clearly this is an equivalence relation on $\mathfrak{M}_p$. Now suppose that $\varphi_1=\psi_1, \varphi_2=\psi_2, \varphi_3=\psi_3$ are formal median identities. Then $\ml\varphi_1\varphi_2\varphi_3\mr = \ml\psi_1\psi_2\psi_3\mr$ is also a formal median identity, so the ternary operator on  $\mathfrak{M}_p$ descends to a ternary operator on the quotient of $\mathfrak{M}_p$ by the equivalence relation $\thicksim_{med}$. Moreover this operator makes the quotient itself a median algebra as it now satisfies the median axioms. Given any other median algebra $(Y, \mu)$ and points $y_1,\ldots,y_p\in Y$ the universal map of ternary algebras from $\mathfrak{M}_p$ to $Y$ factors through the quotient by definition of the formal median identities, so we have shown that:
\begin{lem}The quotient $\mathfrak{M}_p/\thicksim_{med}$ is the free median algebra on $p$ points.
\end{lem}

In the case of coarse median spaces the median identities will only hold up to bounded error. To quantify the errors we need to express the equivalence relation $\thicksim_{med}$ explicitly, which we do by describing it in terms of elementary transformations which admit metric control.

\subsection{Elementary transformations}

We define \emph{elementary transformations} inductively as follows:
\begin{itemize}
  \item For $\varphi \in \mathfrak{S}_p(0)$, the only allowed elementary transformations are of the form  $\varphi \mapsto \ml\varphi\varphi\psi\mr$ for some $\psi \in \mathfrak{M}_p$;
  \item Suppose for elements in $\mathfrak{M}_p(n-1)$, we have defined all the allowed elementary transformations. For $\varphi \in \mathfrak{S}_p(n)$ we define the allowed elementary transformations of $\varphi$ as follows:
      \begin{itemize}
        \item[\emph{Type I).}]  $\varphi \mapsto \ml\varphi\varphi\psi\mr$ for any $\psi$ in $\mathfrak{M}_p$; if $\varphi=\ml\psi\psi\varphi'\mr$, then $\varphi \mapsto \psi$;
        \item[\emph{Type II).}] If $\varphi=\ml\varphi_1\varphi_2\varphi_3\mr$, then $\varphi \mapsto \ml\varphi_{\sigma(1)}\varphi_{\sigma(2)}\varphi_{\sigma(3)}\mr$ for some $\sigma \in S_3$;
        \item[\emph{Type III).}] If $\varphi=\ml\ml\varphi_1\varphi_2\varphi_3\mr\varphi_2\varphi_4\mr$, then $\varphi \mapsto \ml\varphi_1\varphi_2\ml\varphi_3\varphi_2\varphi_4\mr\mr$; \\
            or if $\varphi=\ml\varphi_1\varphi_2\ml\varphi_3\varphi_2\varphi_4\mr\mr$, then $\varphi \mapsto \ml\ml\varphi_1\varphi_2\varphi_3\mr\varphi_2\varphi_4\mr$;
        \item[\emph{Type IV).}] If $\varphi=\ml\varphi_1\varphi_2\varphi_3\mr$, and $\varphi_1'$ is obtained from $\varphi_1$ via an elementary transformation, then $\varphi \mapsto \ml\varphi_1'\varphi_2\varphi_3\mr$ (note that $\varphi_1 \in \mathfrak{M}_p(n-1)$, so the elementary transformations of $\varphi_1$ have already been defined).
      \end{itemize}
\end{itemize}
If $\psi$ is obtained from $\varphi$ by a single elementary transformation we write $\varphi \thicksim_{ET} \psi$. It is easy to see that the relation $\thicksim_{ET}$ is symmetric, i.e. $\varphi \thicksim_{ET} \psi$ if and only if $\psi \thicksim_{ET} \varphi$.

\subsection{A construction for the free median algebra}
Let $\thicksim$ be the transitive closure of $\thicksim_{ET}$ on the set $\mathfrak{M}_p$ of formal ternary expressions. The closure of $\thicksim_{ET}$ is an equivalence relation; we denote the quotient set by $\M_p=\mathfrak{M}_p/\thicksim$, and use $\ql\varphi\qr$ to denote the equivalence class of $\varphi$ in $\M_p$. Furthermore, we define an induced ternary operator $\m$ on $\M_p$ as follows:
$$\m(\ql\varphi_1\qr,\ql\varphi_2\qr,\ql\varphi_3\qr):=\ql\ml\varphi_1\varphi_2\varphi_3\mr\qr$$
for $\varphi_i \in \mathfrak{M}_p$.

We will show that $(\M_p,\m)$ is the free median algebra generated by $[\alpha_1], \ldots, [\alpha_p]$.

\begin{lem}
$\m$ is a well-defined median operator on $\M_p$.
\end{lem}

\begin{proof}
We need to check if $\varphi_1 \thicksim \varphi_1', \varphi_2 \thicksim \varphi_2'$ and $\varphi_3 \thicksim \varphi_3'$, then $\ml\varphi_1\varphi_2\varphi_3\mr \thicksim \ml\varphi_1'\varphi_2'\varphi_3'\mr$.

First suppose that $\varphi_1\thicksim_{ET}\varphi_1'$. Then $\ml\varphi_1\varphi_2\varphi_3\mr \mapsto \ml\varphi_1'\varphi_2\varphi_3\mr$ is a type IV elementary transformation. More generally  if $\varphi\thicksim \varphi_1'$ then a sequence of type IV transformations will ensure that $\ml\varphi_1\varphi_2\varphi_3\mr \thicksim \ml\varphi_1'\varphi_2\varphi_3\mr$. Now in the same way, in addition using type II elementary transformations, we can show that $\ml\varphi_1'\varphi_2\varphi_3\mr \thicksim \ml\varphi_1'\varphi_2'\varphi_3\mr\thicksim\ml\varphi_1'\varphi_2'\varphi_3'\mr
$. Hence $\m$ is a well defined ternary operator on $\M_p$. The fact that it is a median  follows immediately by applying the elementary transformations to its definition.
\end{proof}

\begin{prop}
The median algebra $(\M_p,\m)$ is the free median algebra generated by $\{[\alpha_1],\ldots,[\alpha_p]\}$.
\end{prop}

\begin{proof}
One only needs to check the \emph{universal property}: for any median algebra $(Y,\nu)$ and any map $f: \Omega_p \rightarrow Y$, there exists a unique median homomorphism  $\bar{f}: (\M_p,\m) \rightarrow (Y,\nu)$ extending $f$. Since $\mathfrak{M}_p$ is the free ternary algebra on $p$ points the map $f$ extends uniquely to a  map $\tilde{f}: \mathfrak{M}_p \rightarrow Y$. Explicitly $\tilde{f}(\varphi)=\varphi_Y(f(\alpha_1),\ldots, f(\alpha_p))$.

Now we define $\bar{f}(\ql\varphi\qr):=\tilde{f}(\varphi)$, for any $\ql\varphi\qr \in \M_p$. We need to check that this is a well-defined median homomorphism.

\emph{Well-definedness:} It suffices to show if $\varphi_1\thicksim_{ET}\varphi_2$ for two given formal median expressions, then $\tilde{f}(\varphi_1)=\tilde{f}(\varphi_2)$. We do it via  induction on $\xi(\varphi_1)$. If $\xi(\varphi_1)=0$, we know $\varphi_2=\ml\varphi_1\varphi_1\psi\mr$ for some $\psi \in \mathfrak{M}_p$, which implies
$$\tilde{f}(\varphi_1)=f(\varphi_1)=\nu\big(f(\varphi_1),f(\varphi_1),\tilde{f}(\psi)\big)=\nu\big(\tilde{f}(\varphi_1),\tilde{f}(\varphi_1),\tilde{f}(\psi)\big)=\tilde{f}\big(\ml\varphi_1\varphi_1\psi\mr\big)=\tilde{f}(\varphi_2).$$
If $\varphi_1=\ml\psi_1\psi_2\psi_3\mr$ with $\xi(\psi_i)<\xi(\varphi_1)$, by the definition of elementary transformations, the proof is divided into four cases:
\begin{itemize}
  \item \emph{Type I) to III).}  One just needs to observe that these elementary transformations do not change elements in a median algebra. For example in the case of a Type III) elementary transformation, $\psi_1$ will have the form $\psi_1=\ml\psi_4\psi_2\psi_5\mr$ for some $\psi_4,\psi_5$, so that $\varphi_1=\ml\ml\psi_4\psi_2\psi_5\mr\psi_2\psi_3\mr$, and $\varphi_2$ will have the form $\varphi_2=\ml\psi_4\psi_2\ml\psi_5\psi_2\psi_3\mr\mr$.  Then we have:
      \begin{eqnarray*}
        \tilde{f}(\varphi_1) & = & \nu\big(\tilde{f}(\ml\psi_4\psi_2\psi_5\mr),\tilde{f}(\psi_2),\tilde{f}(\psi_3)\big)\\
        & = & \nu\big(\nu(\tilde{f}(\psi_4),\tilde{f}(\psi_2),\tilde{f}(\psi_5)),\tilde{f}(\psi_2),\tilde{f}(\psi_3)\big)\\
        & = & \nu\big(\tilde{f}(\psi_4),\tilde{f}(\psi_2),\nu(\tilde{f}(\psi_5),\tilde{f}(\psi_2),\tilde{f}(\psi_3))\big)\\
        & = & \nu\big(\tilde{f}(\psi_4),\tilde{f}(\psi_2),\tilde{f}(\ml\psi_5\psi_2\psi_3\mr)\big)\\
        & = & \tilde{f}\Big( \ml\psi_4\psi_2\ml\psi_5\psi_2\psi_3\mr\mr\Big)\\
        & = & \tilde{f}(\varphi_2).
      \end{eqnarray*}
      Here the third equality follows axiom (M3) for a median algebra, while the other equalities all follow from the fact that $\tilde{f}$ is a map of ternary algebras. The  elementary transformations of Type I) and Type II) can be checked similarly.
  \item \emph{Type IV).} $\psi_1'$ is obtained from $\psi_1$ via an elementary transformation and $\varphi_2=\ml\psi_1'\psi_2\psi_3\mr$: by induction one has $\tilde{f}(\psi_1')=\tilde{f}(\psi_1)$, which implies
      $$\tilde{f}(\varphi_1)=\nu\big( \tilde{f}(\psi_1),\tilde{f}(\psi_2),\tilde{f}(\psi_3) \big) = \nu\big( \tilde{f}(\psi_1'),\tilde{f}(\psi_2),\tilde{f}(\psi_3) \big) = \tilde{f}(\varphi_2).$$
\end{itemize}

\emph{Median homomorphism:} By construction, we have:
\begin{eqnarray*}
  \bar{f}\big(\m(\ql\varphi_1\qr,\ql\varphi_2\qr,\ql\varphi_3\qr)\big)&=&\bar{f}\big( \ql\ml\varphi_1\varphi_2\varphi_3\mr\qr \big)=\tilde{f}\big(\ml\varphi_1\varphi_2\varphi_3\mr\big)\\
  &=&\nu\big( \tilde{f}(\varphi_1),\tilde{f}(\varphi_2),\tilde{f}(\varphi_3) \big)=\nu\big(\bar{f}(\ql\varphi_1\qr),\bar{f}(\ql\varphi_2\qr),\bar{f}(\ql\varphi_3\qr)\big).
\end{eqnarray*}

\emph{Uniqueness:} This follows from the definition of $\mathfrak{M}_p$ by an easy inductive argument.
\end{proof}

As a direct corollary of the above construction, we have the following description for formal median identities.
\begin{cor}\label{fmi-char}
Given $\varphi,\psi \in \mathfrak{M}_p$, then $\varphi=\psi$ is a formal median identity if and only if $\varphi \thicksim \psi$.
\end{cor}

\begin{proof}
Since $\M_p$ is a median algebra the equivalence relation $\thicksim$ must contain $\thicksim_{med}$ and the universal map from $\mathfrak{M}_p/\thicksim_{med}$ to the quotient $\M_p$ maps the class of $\varphi$ to the class $[\varphi]$. Since $\M_p$ is also universal this map must be an isomorphism and the equivalence relations agree. Thus $\varphi=\psi$ is a formal median identity if and only if $\varphi \thicksim_{med} \psi$ if and only if $\varphi \thicksim \psi$.
\end{proof}

In \cite{zeidler2013coarse} Zeidler established that, while a given formal median identity does not have to hold for a  coarse median operator, it does hold up to bounded error: for $\mu$ a coarse median on a metric space $(X,d)$ with parameters $\rho,H$ and $\varphi=\psi$ a formal median identity, there exists a constant $R=R(\rho,H,\varphi,\psi)$ such that for all $x_1,\ldots,x_p\in X$,
$$\varphi_X(x_1,\ldots,x_p) \thicksim_R \psi_X(x_1,\ldots,x_p).$$

In order to prove Theorem \ref{simp.cma} we need to establish this under our, \emph{a priori} weaker, axioms $(C0)'\sim (C2)'$.

\begin{prop}\label{strong-Zeidler}
Let $(X,d)$ be a metric space with a ternary operator $\mu$ satisfying (C0)' $\sim$ (C2)' with parameters $\rho,\kappao$ and $\kappaiv$. Let $\varphi=\psi$ be a formal median identity with $p$ variables. Then there exists a constant $R=R(\rho,\kappao,\kappaiv,\varphi,\psi)$, such that for any $x_1,\ldots,x_p \in X$,
$$\varphi_X(x_1,\ldots,x_p) \thicksim_R \psi_X(x_1,\ldots,x_p).$$
\end{prop}

We remark that if $\mu$ is a coarse median on a space $(X,d)$ with parameters $\rho, H$, then (C0)' $\sim$ (C2)' hold with parameters $\rho, \kappao,\kappaiv$ where $\kappao,\kappaiv$ depend only on $\rho$ and $H(4)$, thereby  recovering Zeidler's result, but with the constant depending only on the parameters $\rho,H(4)$. We also note that if we weaken the affine control in (C1)' to a uniform bornology then the result remains true.

\begin{proof}[Proof of \ref{strong-Zeidler}]
Since $\rho,\kappao,\kappaiv$ are fixed throughout the proof, we will omit these parameters from the notation and write $R(\varphi,\psi)$ in place of $R(\rho,\kappao,\kappaiv,\varphi,\psi)$.

By Corollary \ref{fmi-char}, we may choose a finite sequence of elementary transformations:
$$\varphi = \psi_0 \thicksim_{ET} \psi_1 \thicksim_{ET} \ldots \thicksim_{ET} \psi_n=\psi.$$
We will prove the result in the case that $\varphi \thicksim_{ET} \psi$, the general case following from this by adding the constants. By definition, there are four cases:
\begin{itemize}
  \item \emph{Type I) and Type II)}: By (C0)', we have $\varphi_X\thicksim_{\kappao} \psi_X$ so we choose $R(\varphi,\psi)=\kappao$;
  \item \emph{Type III)}. Assume $\varphi=\ml\ml\varphi_1\varphi_2\varphi_3\mr\varphi_2\varphi_4\mr$ and $\psi=\ml\varphi_1\varphi_2\ml\varphi_3\varphi_2\varphi_4\mr\mr$. By (C2)', we have:
      \begin{eqnarray*}
          \varphi_X &=& \mu\big(\mu((\varphi_1)_X,(\varphi_2)_X,(\varphi_3)_X),(\varphi_2)_X,(\varphi_4)_X\big) \\
                    & \thicksim_\kappaiv &\mu\big((\varphi_1)_X,(\varphi_2)_X,\mu((\varphi_3)_X,(\varphi_2)_X,(\varphi_4)_X)\big) \\
                    &=& \psi_X,
      \end{eqnarray*}
so we choose $R(\varphi,\psi)=\kappaiv$.
  \item \emph{Type IV)}. Let $\varphi,\psi\in \mathfrak{M}_p(n)$ and suppose that $\varphi$ is equivalent to $\psi$ by an elementary transformation of Type IV). Thus $\varphi=\ml\varphi_1\varphi_2\varphi_3\mr$, $\psi=\ml\varphi_1'\varphi_2\varphi_3\mr$ with $\varphi_1,\varphi_1'\in\mathfrak{M}_p(n-1)$ and  $\varphi_1 \thicksim_{ET} \varphi_1'$. If the latter is an elementary transformation of type I), II) or III) then we know that there is a constant $R(\varphi_1,\varphi_1')$ such that $(\varphi_1)_X\thicksim_{R(\varphi_1,\varphi_1')} (\varphi_1')_X$. Otherwise we have a further type IV) elementary transformation, and by induction on $n$ we may again assume that $(\varphi_1)_X\thicksim_{R(\varphi_1,\varphi_1')}  (\varphi_1')_X$, the base case $n=0$ being vacuous. By (C1)', we have:
      $$\varphi_X=\mu((\varphi_1)_X,(\varphi_2)_X,(\varphi_3)_X) \thicksim_{\rho(R(\varphi_1,\varphi_1'))} \mu((\varphi_1')_X,(\varphi_2)_X,(\varphi_3)_X) = \psi_X$$
      and we set $R(\varphi,\psi)=\rho(R(\varphi_1,\varphi_1'))$.
\end{itemize}
\end{proof}

\subsection{Proof of Theorem \ref{simp.cma}}

We know that (C0)' $\sim$ (C2)' hold in a metric space $(X,d)$ with a coarse median $\mu$. Conversely we must prove that if (C0)' $\sim$ (C2)' hold then (C2) also holds, since (C1) follows easily from (C0)' and (C1)'.

For each $p$, we choose a splitting $\M_p\to \mathfrak{M}_p$ of the quotient map $\mathfrak{M}_p\to \M_p$, such that $[\alpha_i]$ is lifted to $\alpha_i$. In other words, for each $p$ and each element of $\M_p$ we choose a representative in $\mathfrak{M}_p$. In a median algebra, the median stabilisation of finitely many points is itself finite, \cite[Lemma 6.20]{van1993theory}, so $\M_p$ is finite and thus the splittings provide finite subsets $\Theta_p$ of $\mathfrak{M}_p$.

Now let $A$ be a finite subset of $X$, with cardinality $p$, and enumerate the set $A$ as $A=\{x_1,\ldots,x_p\} $. Define $\pi: A \rightarrow \M_p$ by $x_i \mapsto [\alpha_i]$.
We consider the set of formal median identities of the form $\ml\varphi_1\varphi_2\varphi_3\mr=\psi$ where $\varphi_1,\varphi_2,\varphi_3,\psi\in \Theta_p$. Since $\Theta_p$ is finite there are only finitely many of these. By Proposition \ref{strong-Zeidler} for each of these formal median identities there is a constant $R$, which does not depend on $A$, such that $\mu((\varphi_1)_X,(\varphi_2)_X,(\varphi_3)_X) \thicksim_{R} \psi_X$. Let $H_p$ be an upper bound for the (finitely many) constants $R$ arising from these identities. (Note moreover that the constant $H_p$ does depend on the parameters $\rho, \kappao, \kappaiv$ but not on the space $X$ itself.)

Now we define a map $\lambda: \M_p \rightarrow X$ by $\ql\varphi\qr \mapsto (\varphi)_X$ for $\varphi\in\Theta_p$. By the above estimate for $\varphi_1,\varphi_2,\varphi_3\in\Theta_p$, we have:
\begin{eqnarray*}
  \lambda\big( \m(\ql\varphi_{1}\qr,\ql\varphi_{2}\qr,\ql\varphi_{3}\qr) \big) &=& \lambda \big( \ql\ml\varphi_{1}\varphi_{2}\varphi_{3}\mr\qr \big) \thicksim_{H_p} \mu\big((\varphi_{1})_X,(\varphi_{2})_X,(\varphi_{3})_X\big)\\
  &=&\mu\big( \lambda(\ql\varphi_{1}\qr),\lambda(\ql\varphi_{2}\qr),\lambda(\ql\varphi_{3}\qr) \big).
\end{eqnarray*}
Clearly, $\lambda\pi$ is the inclusion $A \hookrightarrow X$, so we are done.
\qed

\section{Characterising rank in a coarse median space}\label{rank}
For a coarse median space the rank is determined by the dimensions of the approximating median algebras for its finite subsets. However, replacing Bowditch's approximation axiom (C2) by our axiom (C2)' controlling four point approximations, it is no longer sufficient to consider the dimensions of these: the free median algebra on four points has dimension 3 (see Figure \ref{fig:freemedian}) which clearly does not provide a universal bound on the rank.

We begin by analysing the case of rank 1 coarse median spaces where we reprove Bowditch's result that for geodesic spaces this implies hyperbolicity. We will do this directly, without recourse to the asymptotic cones that are an essential ingredient in Bowditch's proof. We will then consider the general case, showing how to define ``coarse cubes'' in terms of the coarse median, showing that a coarse median space has rank at most $n$ if and only if it cannot contain arbitrarily large $(n+1)-$ coarse cubes.

\subsection{Rank 1 coarse median spaces}
According to Bowditch \cite{bowditch2013coarse}, the class of geodesic coarse median spaces with rank 1 coincides with the class of geodesic $\delta$-hyperbolic spaces. Given a geodesic hyperbolic space it is possible to construct a coarse median using the thinness of geodesic triangles; for the opposite direction Bowditch used the method of asymptotic cones. Here we give another more direct way to prove this.

We will use the following characterisation of Gromov's hyperbolicity for geodesic spaces, developed by Papasoglu and Pomroy, which asserts that for a geodesic metric space stability of quasi-geodesics implies hyperbolicity.
\begin{thm}[\cite{papasoglu1995strongly,pomroy,chatterji2006characterization}]\label{char for hyp}
Let $(X,d)$ be a geodesic metric space. Suppose there exists $\varepsilon>0, R>0$, such that for any $a,b\in X$, all $(1, \varepsilon)$-quasi-geodesics from $a$ to $b$ remain in an $R$-neighbourhood of any geodesic connecting $a,b$. Then $X$ is hyperbolic.
\end{thm}

We will prove that in a rank 1 geodesic coarse median space quasi-geodesics are stable in the above sense. The key idea is to use coarse intervals in place of the geodesics connecting $a$ and $b$. We will prove the following:

\begin{thm}\label{rank1}
Let $(X,d,\mu)$ be a geodesic coarse median space with rank $1$. For any $\zeta, \varepsilon >0$, there exists a constant $\widetilde{D}=\widetilde{D}(\zeta,\varepsilon)$, such that for any $a,b \in X$ and any $(\zeta,\varepsilon)$-quasi-geodesic $\gamma \colon [s,t] \rightarrow X$ connecting $a$ and $b$, the Hausdorff distance between $[a,b]$ and $\mathrm{Im}(\gamma)$ is less than $\widetilde{D}$. \end{thm}

Applying Theorem \ref{char for hyp} we obtain the following corollary.

\begin{cor}
Let $(X,d,\mu)$ be a geodesic coarse median space with rank $1$.  Then $(X,d)$ is hyperbolic.
\end{cor}

\begin{proof}
According to Theorem \ref{rank1}, any $(1,\varepsilon)$ quasi-geodesic is within Hausdorff distance $\widetilde D$ of the interval $[a,b]$. In particular any geodesic from $a$ to $b$ is within Hausdorff distance $\widetilde D$ of the interval. We now apply Theorem \ref{char for hyp} with $R=2\widetilde D$.
\end{proof}

We will also give a characterisation of rank $1$ coarse median spaces in terms of the geometry of intervals. Note that here we do not require the space to be (quasi)-geodesic.

\begin{thm}\label{char for rank1}
Let $(X,d,\mu)$ be a coarse median space. Then $X$ has rank at most $1$ if and only if there exists a constant $\lambda$ such that for any $a,b,x\in X$ with $x\in [a,b]$, we have:
\begin{equation}\label{thintervaltriangles}
[a,b] \subseteq \mathcal{N}_\lambda([a,x]) \cup \mathcal{N}_\lambda([x,b]).
\end{equation}
\end{thm}

We regard the inclusion (\ref{thintervaltriangles}) as an analog of Gromov's thin triangles condition for coarse intervals, and begin by proving that, moreover, in a rank $1$ space an analogous condition holds for neighbourhoods of intervals. 

\begin{lem}\label{hyp trg for intvl}
Let $(X,d,\mu)$ be a coarse median space with rank 1 achieved under parameters $\rho,H$. Then for any $\zeta\geqslant 0$, there exists a constant $\zeta'=2H(4)+\zeta$ such that for any $a,b,c\in X$, we have:
$$\mathcal{N}_{\zeta}([a,b]) \subseteq \mathcal{N}_{\zeta'}([a,c]) \cup \mathcal{N}_{\zeta'}([c,b]).$$
\end{lem}

\begin{proof}
For any $y\in \mathcal{N}_{\zeta}([a,b])$, there exists some $x\in X$, such that $\mu(a,b,x) \thicksim_\zeta y$, and we set $A=\{a,b,c,x\}$.  By Definition \ref{def for cma} and Definition \ref{convention} there exists a finite rank 1  median algebra $(\Pi, \mu_{\Pi})$ (i.e., a tree) and maps $\pi\colon A \rightarrow \Pi$, $\lambda\colon  \Pi \rightarrow X$ satisfying the conditions in (C2); furthermore $\lambda \circ \pi$ is the inclusion $A \hookrightarrow X$. Denote $\pi(a), \pi(b), \pi(c), \pi(x)$ by $\overline{a}, \overline{b}, \overline{c}, \overline{x}$ respectively.

Set $\overline{m}=\mu_{\Pi}(\overline{a},\overline{b},\overline{x})$, then $\mu_{\Pi}(\overline{a},\overline{b},\overline{m})=\overline{m}$. Since $\Pi$ is a tree, we have:
$$\mu_{\Pi}(\overline{a},\overline{c},\overline{m})=\overline{m} \quad\mbox{or}\quad \mu_{\Pi}(\overline{b},\overline{c},\overline{m})=\overline{m}.$$
Without loss of generality, assume the former, so
$$\lambda(\overline{m})=\lambda(\mu_{\Pi}(\overline{a},\overline{c},\overline{m})) \thicksim_{H(4)} \mu(a,c,\lambda(\overline{m})).$$
On the other hand, by the definition of $\overline{m}$, we have:
$$\lambda(\overline{m}) = \lambda(\mu_{\Pi}(\overline{a},\overline{b},\overline{x})) \thicksim_{H(4)} \mu(a,b,x) \thicksim_\zeta y.$$
Combine the above two estimates:
$$y \thicksim_{2H(4)+\zeta} \mu(a,c,\lambda(\overline{m})),$$
which implies $y\in \mathcal{N}_{2H(4)+\zeta}([a,c])$. Finally, take $\zeta'=2H(4)+\zeta$, then the lemma holds.
\end{proof}

Recall that in a hyperbolic space, any point on a geodesic from $a$ to $b$ sits logarithmically far, with respect to path-length,  from any path connecting $a$ and $b$. The analogous statement holds for rank $1$ spaces, if one replaces geodesics with intervals:
\begin{prop}\label{control}
Let $(X,d,\mu)$ be a geodesic coarse median space with rank 1, $\gamma$ a rectifiable path in $X$ connecting $a,b \in X$ and $\ell(\gamma)$ the length of $\gamma$, then there exist constants $C_1,C_2$ such that
$$[a,b] \subseteq \mathcal{N}_{C_1\log_2\ell(\gamma)+C_2}(\mathrm{Im}~\gamma).$$
\end{prop}

\begin{proof}
Let $K,H_0,H$ be parameters of $(X,d,\mu)$ under which $\rank X \leqslant 1$ can be achieved. Without loss of generality we will assume that $\gamma$ has been reparameterised to have unit speed. At the cost of varying the constants $C_1, C_2$ we can simplify the argument, again without loss of generality, and  assume $\ell(\gamma)=2^N$ for some integer $N$. We will argue by induction on $N$.

$N=0$: $\ell(\gamma)=1$, which implies $d(a,b)\leqslant 1$. For any $c\in [a,b]$, there exists some $x\in X$ such that $c=\mu(a,b,x)$. Now by axiom (C1), we have:
$$c=\mu(a,b,x)\thicksim_{K+H_0}\mu(a,a,x)=a,$$
which implies $c\in B(a,K+H_0)$. So $[a,b] \subseteq \mathcal{N}_{K+H_0}(\gamma)$ for $\ell(\gamma)=1$. We set $s_1=K+H_0$.

Induction step: Assume that we have established that for all pairs $a,b\in X$ and rectifiable paths $\gamma'$ connecting them of length $2^{N-1}$, there is a constant $s_{N-1}$ such that  $[a,b] \subseteq \mathcal{N}_{s_{N-1}}(\mathrm{Im}~\gamma')$. Now consider a path $\gamma$ of $\ell(\gamma)=2^N$ joining $a$ to $b$. We will denote  the midpoint of $\gamma$ by $c=\gamma(2^{N-1})$. By Lemma \ref{hyp trg for intvl}, we have:
$$[a,b] \subseteq \mathcal{N}_{2H(4)}([a,c]) \cup \mathcal{N}_{2H(4)}([c,b]).$$
By our inductive hypothesis, we have:
$$[a,c] \subseteq \mathcal{N}_{s_{N-1}}(\gamma([0,2^{N-1}])) \quad \mbox{and} \quad [c,b] \subseteq \mathcal{N}_{s_{N-1}}(\gamma([2^{N-1},2^N])),$$
which implies $[a,b] \subseteq \mathcal{N}_{2H(4)}([a,c]) \cup \mathcal{N}_{2H(4)}([c,b]) \subseteq \mathcal{N}_{s_{N-1}+2H(4)}(\mathrm{Im}~\gamma)$, i.e. we can take $s_N=s_{N-1}+2H(4)$, in other words, $s_N=K+H_0+(N-1)\cdot 2H(4)$.
In conclusion, take $C_1=2H(4)$ and $C_2=K+H_0-2H(4)$, we have $[a,b] \subseteq \mathcal{N}_{C_1\log_2\ell(\gamma)+C_2}(\mathrm{Im}~\gamma)$.
\end{proof}

As in the hyperbolic case, we need the following lemma to tame quasi-geodesics, replacing an arbitrary quasi-geodesic by a rectifiable quasi-geodesic close to the original.

\begin{lem}[\cite{BH99}]\label{taming geodesics}
Let $(X,d)$ be a geodesic metric space. Given any $(\zeta,\varepsilon)$-quasi-geodesic $\gamma \colon [s,t] \rightarrow X$, one can find a continuous $(\zeta,\varepsilon')$-quasi-geodesic $\gamma' \colon [s,t] \rightarrow X$ such that
\begin{enumerate}[1)]
  \item $\gamma(s)=\gamma'(s)$, $\gamma(t)=\gamma'(t)$;
  \item $\varepsilon'=2(\zeta+\varepsilon)$;
  \item $\ell(\gamma'|_{[s',t']}) \leqslant k_1 d(\gamma'(s'),\gamma'(t'))+k_2$ for all $s',t'\in [s,t]$, where $k_1=\zeta(\zeta+\varepsilon)$ and $k_2=(\zeta \varepsilon'+3)(\zeta+\varepsilon)$;
  \item the Hausdorff distance between $\mathrm{Im}\gamma$ and $\mathrm{Im}\gamma'$ is less than $\zeta+\varepsilon$.
\end{enumerate}
\end{lem}

We also need the following two lemmas which hold trivially in the median case. Recall in a coarse median space with given parameters $\rho,H$, we defined the constant $\kappav=\rho(H(5))+\rho(2H(5))+2H(5)$.

\begin{lem}\label{coarse estimate}
Let $(X,d,\mu)$ be a coarse median space with parameters $\rho,H$. Then there exists a constant $C=\rho(\kappav)+\kappav$ such that for all $a,b\in X$, $x \in [a,b]$, $y\in [a,x]$ and $z\in [x,b]$, we have $x\in\mathcal{N}_C([y,z])$.
\end{lem}

\begin{proof}
By definition there exist $w,v,v'\in X$ such that $x=\mu(a,w,b)$, $y=\mu(a,v,x)$ and $z=\mu(x,v',b)$. Now
$$\mu(y,w,b)=\mu(\mu(a,v,x),w,b) \thicksim_\kappav \mu(\mu(a,w,b),\mu(v,w,b),x)=\mu(x,\mu(v,w,b),x)=x,$$
which implies
\begin{eqnarray*}
\mu(y,w,z) &=& \mu(y,w,\mu(x,v',b)) \quad \thicksim_\kappav \quad  \mu(x,\mu(y,w,v'),\mu(y,w,b)) \\
  & \thicksim_{\rho(\kappav)}&  \mu(x,\mu(y,w,v'),x) \quad = \quad x.
\end{eqnarray*}
In other words, $x\in \mathcal{N}_{\rho(\kappav)+\kappav}(\{\mu(y,w,z)\})\subset \mathcal{N}_{\rho(\kappav)+\kappav}([y,z])$. Take $C=\rho(\kappav)+\kappav$ and the result holds.
\end{proof}

The next lemma generalises Lemma \ref{redefine cma on intvl} to the context of coarse neighbourhoods of intervals $\mathcal{N}_\lambda([a,b])$.
\begin{lem}\label{estimate}
Let $(X,d,\mu)$ be a coarse median space with parameters $\rho,H$, let $\lambda$ be a positive constant and $a,b$ be points in $X$. Then $x\in \mathcal{N}_\lambda([a,b])$ implies that $\mu(a,b,x)\thicksim_{\rho(\lambda)+\lambda+\kappav} x$.
\end{lem}

\begin{proof}
By definition, $x\thicksim_\lambda x'\in[a,b]$, which implies there exists $z\in X$ such that $x'=\mu(a,b,z)$. So
$$\mu(a,b,x)\thicksim_{\rho(\lambda)}\mu(a,b,\mu(a,b,z))\thicksim_\kappav\mu(a,b,z)=x'\thicksim_\lambda x.$$
\end{proof}

\begin{proof}[Proof of Theorem \ref{rank1}]
Take a $(\zeta, \varepsilon)$-quasi-geodesic connecting points $a$ and $b$. By Lemma \ref{taming geodesics}, we may, at the cost of moving the quasi-geodesic at most $\zeta+\varepsilon$, obtain a $(\zeta,\varepsilon)$-quasi-geodesic $\gamma$ which is continuous and such that there exist constants $k_1$ and $k_2$ depending only on $\zeta, \varepsilon$, such that for any $s',t' \in [s,t]$, $\ell(\gamma|_{[s',t']}) \leqslant k_1 d(\gamma(s'),\gamma(t'))+k_2$. Let $K,H_0,H$ be parameters of $(X,d,\mu)$ under which $\rank X \leqslant 1$ can be achieved.

We will first show that the interval $[a,b]$ sits within a uniformly bounded neighbourhood of $\gamma$.

Let $D=\sup \{d( x, \mathrm{Im}\gamma): {x\in[a,b]} \}$, which is finite since it is bounded by $\sup\{d(x,a)\mid x\in [a,b]\}\leq K d(a,b)+H_0$. Indeed for $x=\mu(a,b,c)\in [a,b]$, we have
$$x=\mu(a,b,c) \thicksim_{Kd(a,b)+H_0}\mu(a,a,c)=a.$$

Let $x_0$ be a point in $[a,b]$ such that $d(x_0, \mathrm{Im}~\gamma)>D-1$. We have $\mu(a,b,x_0)\thicksim_{\kappav} x_0$ by Lemma \ref{kappav}. Now if $y\in [a,x_0]$ then $y=\mu(x_0,a,z)$ for some $z\in X$, so
\[
\mu(b,a,y)=\mu(b,a, \mu(x_0,a,z))\thicksim_\kappaiv \mu(\mu(b,a,x_0),a,z)\thicksim_{K\kappav+H_0}\mu(x_0,a,z)=y.
\]
Hence $y\in \mathcal N_\lambda([a,b])$ where $\lambda=K\kappav+\kappaiv+H_0$.

\emph{Claim:} If $d(x_0,a) \geqslant 2D+\lambda$, then there exists some $y \in [x_0,a]$ such that $2D+\lambda \leqslant d(x_0,y) \leqslant 2D+\lambda+K+H_0$.

\emph{Proof of Claim:} First take a geodesic $\gamma':[0, d(x_0,a)]\rightarrow X$ from $x_0$ to $a$. We will approximate $\gamma'$ by a discrete geodesic as follows. Let $x_i=\gamma'(i)$ for $i=0, \dots,  k=\lfloor d(a,x_0)\rfloor$ and let $y_i$ denote the projection of $x_i$ into the interval $[x_0,a]$, i.e., $y_i=\mu(a,x_i,x_0)$. Then for each $i$, $d(y_i,y_{i+1}) \leqslant K+H_0$. If $d(y_j,x_0) \geqslant 2D+\lambda$, for some $j$ then let $i$ be the first $j$ for which this occurs and set $y=y_i$. Since $d(y_{i-1},x_0) < 2D+\lambda$ and $d(x_{i-1}, x_i)=1$ we have $d(x_0,y) =d(x_0, y_i)<2D+\lambda+K+H_0$. Otherwise set $y=a$ and note that $d(y_{k},x_0) \leqslant 2D+\lambda$ so again $d(x_0,y) =d(x_0, a)\leqslant 2D+\lambda+K+H_0$. This completes the proof of the claim.

As shown above $y\in \mathcal N_\lambda([a,b])$, so we can find a point $y'\in [a,b]$ such that $d(y',y)\leqslant \lambda$. Hence by the claim, we have $2D\leqslant d(y',x_0)\leqslant 2D+2\lambda+K+H_0$.

If on the other hand $d(x_0,a) < 2D+\lambda$, then we take $y=y'=a$ and we have $D-1< d(y',x_0)\leqslant 2D+\lambda$.

Repeating the argument with $b$ in place of $a$, we can find points $z\in [x_0,b], z'\in [a,b]$ with: $2D+\lambda \leqslant d(x_0,z) \leqslant 2D+\lambda+K+H_0$ and $d(z,z')\leqslant \lambda$ if $d(x_0,b) \geqslant 2D+\lambda$; $z=z'=b$ with $D-1< d(z',x_0)\leqslant 2D+\lambda$ otherwise.

If $y'=a$ then we set $y''=a$, otherwise by the definition of $D$, there exists a point $y''\in \mathrm{Im}\gamma$ such that $d(y'',y') \leqslant D$, and we can choose a geodesic segment $\alpha$ from $y'$ to $y''$. Similarly if $z'=b$ then we set $z''=b$, otherwise there exists a point $z''\in \mathrm{Im}\gamma$ such that $d(z'',z') \leqslant D$, and we choose a geodesic segment $\beta$ from $z''$ to $z'$.

Now consider the path $\vartheta$ from $y'$ to $z'$ that transverses $\alpha$ then follows $\gamma$ from $y''$ to $z''$, then transverses $\beta$ from $z''$ to $z'$, see Figure \ref{fig:rank1}. We have:
\begin{eqnarray*}
d(y'',z'')&\leqslant& d(y'',y')+d(y',x_0)+d(x_0,z')+d(z',z'') \leqslant2D+2\cdot (2D+2\lambda+K+H_0)\\
&=&6D+4\lambda+2(K+H_0),\\
\end{eqnarray*}
which implies, by Lemma \ref{taming geodesics}, that $\ell(\vartheta) \leqslant 2D+ k_1\cdot(6D+4\lambda+2(K+H_0))+k_2=(6k_1+2)D+2k_1(2\lambda+K+H_0)+k_2$.
\begin{figure}[htbp]
  \centering
  \includegraphics[width=12cm]{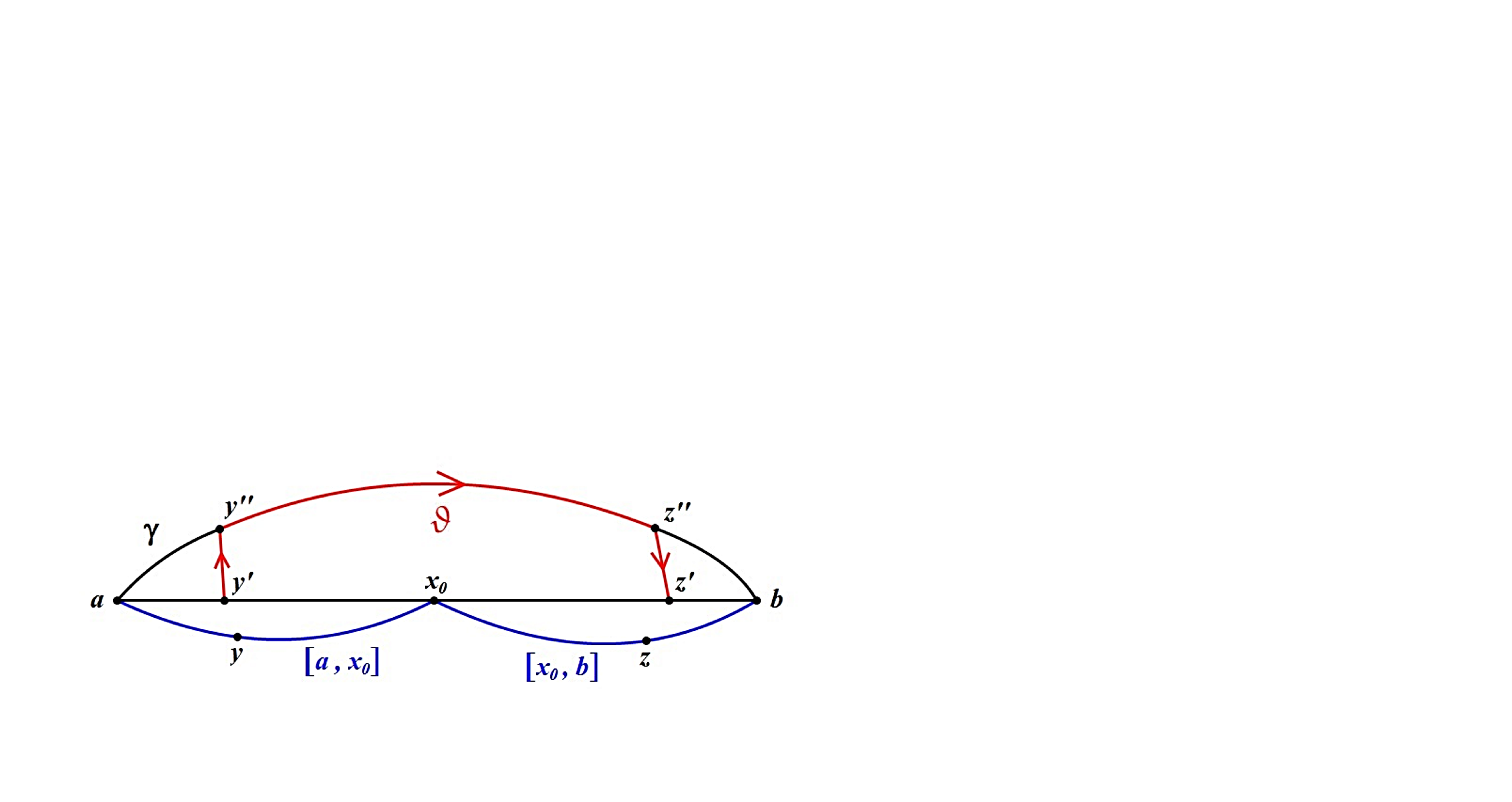}
  \caption{Construction of $\vartheta$} \label{fig:rank1}
\end{figure}

Independent of our choice of the points $y'', z''$, the distance from $x_0$ to any point on $\gamma$ is at least $D-1$. If $y'=a$ then we have set $y''=y'$, so the distance from $x_0$ to the (only) geodesic arc from $y'$ to $y''=y'$ is $d(x_0,a) > D-1$. If on the other hand $y'\not=a$, then by hypothesis $d(x_0,y')\geqslant 2D$, while $d(y',y'')\leqslant D$ so the distance from $x_0$ to the chosen geodesic arc from $y'$ to $y''$ must be at least $D$. A similar argument shows that the distance from $x_0$ to the chosen geodesic arc from $z''$ to $z'$ must be at least $D$ as well, so that $d(x_0, \mathrm{Im}\vartheta) > D-1$.

We would like to use Proposition \ref{control} to obtain an upper bound for $D$, but $x_0$ might not sit inside $[y',z']$. However, by Lemma \ref{coarse estimate}, we know there exists a constant $C$ such that $x_0\in\mathcal{N}_C([y,z])$, which implies there is a point $\mu(y, z, w)\in [y,z]$ which is $C$-close to $x_0$. Now $\mu(y,z,w)\thicksim_{2K\lambda+H_0}\mu(y',z',w)\in [y',z']$ by axiom (C1). Set $x'=\mu(y',z',w)$ and $C'=C+2K\lambda+H_0$ so that $x_0\thicksim_{C'} x'\in [y',z']$.

Hence by Proposition \ref{control},
\begin{eqnarray*}
D-1 &<& d(x_0, \mathrm{Im}\vartheta) \leqslant d(x', \mathrm{Im}\vartheta)+C' \leqslant C_1\log_2\ell(\vartheta)+C_2+C'\\
  &\leqslant& C_1\log_2((6k_1+2)D+2k_1(2\lambda+K+H_0)+k_2)+C_2+C'.
\end{eqnarray*}
The right hand side of the inequalities grows logarithmically fast with respect to $D$, hence there exists some constant $D'=D'(\zeta,\varepsilon)$ such that $D \leqslant D'$. In other words, $[a,b] \subseteq \mathcal{N}_{D'}(\mathrm{Im}\gamma)$.

Now we will show that the quasi-geodesic $\gamma$ sits within a uniformly bounded neighbourhood of the interval  $[a,b]$. Assume there exists some point $x'\in\mathrm{Im}\gamma$ such that $d(x',[a,b]) > D'$. Take a maximal non-empty subinterval $(s_1,t_1) \subseteq [s,t]$ such that $\gamma|_{(s_1,t_1)}$ sits outside $\mathcal{N}_{D'}([a,b])$. As in the first part of the proof, pick a discrete geodesic $a=x_0,x_1,\ldots,x_n=b$ from $a$ to $b$ and set $y_i=\mu(a,b,x_i)$. For each $i$, $y_i\in [a,b]$ and the distance $d(y_i, \mathrm{Im}~\gamma)\leqslant D'$. So either $d(y_i,\mathrm{Im}\gamma|_{[s,s_1]}) \leqslant D'$ or $d(y_i,\mathrm{Im}\gamma|_{[t_1,t]}) \leqslant D'$. Note that $d(y_0,\mathrm{Im}\gamma|_{[s,s_1]}) \leqslant D'$ and $d(y_n,\mathrm{Im}\gamma|_{[t_1,t]}) \leqslant D'$. Take the first $i$ satisfying $d(y_i,\mathrm{Im}\gamma|_{[s,s_1]}) \leqslant D'$ and $d(y_{i+1},\mathrm{Im}\gamma|_{[t_1,t]}) \leqslant D'$, and set $w=y_i$. Then it follows that there exists $s'_1\in[s,s_1]$ and $t'_1\in[t_1,t]$ such that $d(w,\gamma(s'_1)) \leqslant D'$, $d(w,\gamma(t'_1)) \leqslant D'+H_0+K$, which implies $d(\gamma(s'_1),\gamma(t'_1)) \leqslant 2D'+H_0+K$. Since $\gamma$ is assumed to be tame we have $\ell(\gamma|_{[s'_1,t'_1]})\leqslant 2k_1(D'+H_0+K)+k_2$, which implies $\ell(\gamma|_{[s_1,t_1]})\leqslant k_1(2D'+H_0+K)+k_2$. So we have $\mathrm{Im}\gamma \subseteq \mathcal{N}_{k_1(2D'+H_0+K)+k_2+D'}([a,b])$. Finally, take $\widetilde{D}=k_1(2D'+H_0+K)+k_2+D'+\zeta+\varepsilon$, then the Hausdorff distance between $[a,b]$ and the original $(\zeta, \varepsilon)$-quasi-geodesic is controlled by $\widetilde{D}$.
\end{proof}

In our proof of Theorem \ref{char for rank1} we will make use of the following generalisation of Zeidler's parallel edge lemma, \cite[Lemma 2.4.5]{zeidler2013coarse}. The proof in this more general context is more or less identical, and is therefore omitted.

\begin{lem}\label{zeidler}
Let $(X,d)$ be a metric space with a ternary operator $\mu$ satisfying the weak form of axiom (C1) with (arbitrary) control parameter $\rho$. Let $C$ be a CAT(0) cube complex, and $f \colon (C^{(0)},\mu_C) \rightarrow (X,\mu)$ an $L$-quasi-morphism, then for any parallel edges $(x,y),(x',y')$ in $C$, we have
$$d(f(x'),f(y')) \leqslant \rho(d(f(x),f(y)))+2L.$$\qed
\end{lem}

\begin{proof}[Proof of Theorem \ref{char for rank1}]
Necessity is a special case of Lemma \ref{hyp trg for intvl}, so we just need to show sufficiency. Let $\rho,H$ be parameters of $(X,d,\mu)$. By definition and Lemma \ref{median rank assump} for any $p\in \mathbb{N}$ and $A \subseteq X$ with $|A|=p$, there exists a finite median algebra $(\Pi, \mu_{\Pi})$, and maps $\pi \colon A \rightarrow \Pi$, $\sigma \colon \Pi \rightarrow X$ satisfying axioms (C1), (C2), $\langle \pi(A) \rangle=\Pi$ and $\sigma \pi = i_A$. Denote the associated finite CAT(0) cube complex also by $\Pi$. Now we want to modify $\Pi,\pi,\sigma$ so that $\Pi$ has dimension 1.

If $\dim \Pi=1$, nothing needs to be modified. Assume $\dim \Pi \geqslant 2$, i.e. there exists a 2-square $\bar{a},\bar{b},\bar{c},\bar{d}$ in $\Pi$. To be more explicit, assume $\bar{a}$ is connected to $\bar{b}$ and $\bar{c}$ in $\Pi$. Denote $\sigma(\bar{a}),\sigma(\bar{b}),\sigma(\bar{c}),\sigma(\bar{d})$ by $a,b,c,d$. By axiom (C2), we have:
$$b':=\mu(a,b,d)=\mu(\sigma(\bar{a}),\sigma(\bar{b}),\sigma(\bar{d}))\thicksim_{H(p)} \sigma(\mu(\bar{a},\bar{b},\bar{d}))=\sigma(\bar{b})=b.$$
Similarly, $c':=\mu(a,c,d) \thicksim_{H(p)} c$. Since $b'=\mu(a,b,d)\in[a,d]$, by assumption, there exists a constant $\lambda$ such that $[a,d] \subseteq \mathcal{N}_\lambda([a,b']) \cup \mathcal{N}_\lambda([b',d])$. Now since $c'\in [a,d]$, without loss of generality, we can assume that $c'\in \mathcal{N}_\lambda([a,b'])$, which implies $c'\thicksim_{\rho(\lambda)+\lambda+\kappav}\mu(a,b',c')$ by Lemma \ref{estimate}. So
$$c\thicksim_{H(p)}\mu(a,c,d)\thicksim_{\rho(\lambda)+\lambda+\kappav}\mu(a,\mu(a,b,d),\mu(a,c,d))\thicksim_\kappav \mu(a,d,\mu(a,b,c)).$$
On the other hand, $\mu(a,b,c)=\mu(\sigma(\bar{a}),\sigma(\bar{b}),\sigma(\bar{c}))\thicksim_{H(p)}\sigma(\mu(\bar{a},\bar{b},\bar{c}))=a$, so $$\mu(a,d,\mu(a,b,c))\thicksim_{\rho(H(p))}\mu(a,d,a)=a.$$
To sum up, we showed $c$ is $[\rho(H(p))+\rho(\lambda)+H(p)+\lambda+2\kappav)]$-close to $a$. For brevity, if we take $\alpha'(t)=\rho(t)+t+\rho(\lambda)+\lambda+2\kappav$, then $c$ is $\alpha'(H(p))$-close to $a$.

By Lemma \ref{zeidler}, we obtain another function $\alpha(t)=\rho(\alpha'(t))+2t$ such that for any edge $(\bar{x},\bar{y})$ in $\Pi$ parallel to $(\bar{a},\bar{c})$, $\sigma(\bar{x})$ is $\alpha (H(p))$-close to $\sigma(\bar{y})$. Note that the function
$$\alpha(t)=\rho(\alpha'(t))+2t=\rho(\rho(t)+t+\rho(\lambda)+\lambda+2\kappav)+2t$$
depends only on $\rho,\kappav,\lambda$. In particular it does not depend on the set $A$.

Now consider the quotient CAT(0) cube complex $\Pi'$ of $\Pi$ determined by all the hyperplanes except the hyperplane $h$ crossing $(\bar{a},\bar{c})$ (see \cite{chatterji2005wall,nica2004cubulating}). Let $pr \colon \Pi \rightarrow \Pi'$ be the natural projection, and choose a section $s \colon\Pi' \rightarrow \Pi$, i.e. $pr \circ s=id_{\Pi'}$. For any vertex $\bar v$ not adjacent to $h$ we have $s\circ pr(\bar v) = \bar v$, while if $\bar v$ is adjacent to $h$ then either $s\circ pr(\bar v) = \bar v$ or there is an edge crossing $h$ joining $s\circ pr(\bar v),\bar v$. The map $s$ is not uniquely determined, but for $s,s'$ two different sections, according to the analysis above the compositions $\sigma\circ s,\sigma\circ s'$ are $\alpha(H(p))$-close.
Now consider the diagram

\begin{displaymath}
\xymatrix@R=0.5cm{
                &  ~ &      X      \\
  \Pi' \ar@/^/[r]^{s} & \Pi \ar[ur]^{\sigma}  \ar@/^/[l]^{pr}                \\
                &  ~  &    A  \ar[uu]^{i} \ar[ul]_{\pi}              }
\end{displaymath}
and take $\pi'=pr \circ \pi$, $\sigma'=\sigma \circ s$. We have $\sigma' \circ \pi'\thicksim_{\alpha(H(p))}i_A$, and for any $x',y',z'\in \Pi'$, we want to estimate the distance between $\sigma'(\mu_{\Pi'}(x',y',z'))$ and $\mu(\sigma'(x'),\sigma'(y'),\sigma'(z'))$.

\textbf{Claim.} $\mu_{\Pi'}(x',y',z')=pr(s(\mu_{\Pi'}(x',y',z')))=pr(\mu_{\Pi}(s(x'),s(y'),s(z')))$.

Indeed, by the construction of $\Pi'$, for any hyperplane $h' \neq h$ of $\Pi$ and $\bar{u},\bar{v}\in \Pi$, $h'$ separates $\bar{u},\bar{v}$ if and only if $h'$, as a hyperplane of $\Pi'$, separates $pr(\bar{u}),pr(\bar{v})$. Now assume that
\begin{equation}\label{assumption1}
 pr(s(\mu_{\Pi'}(x',y',z')))\neq pr(\mu_{\Pi}(s(x'),s(y'),s(z'))),
 \end{equation}
 then there exists a hyperplane $h'$ of $\Pi'$ such that $h'$ separates $pr(s(\mu_{\Pi'}(x',y',z')))$ and $pr(\mu_{\Pi}(s(x'),s(y'),s(z')))$, which implies $h'$, as a hyperplane of $\Pi$ rather than $h$, separates $s(\mu_{\Pi'}(x',y',z'))$ and $\mu_{\Pi}(s(x'),s(y'),s(z'))$. We select a vertex $\bar v_0\in \Pi$ such that $s\circ pr(\bar v_0)=\bar v_0$ and let $h'_+$ denote the half space corresponding to the hyperplane $h'$ containing $\bar v_0$ and $h'_-$ the complementary halfspace. In $\Pi'$ we will abuse notation and also denote by $h'_+$ the halfspace containing $pr (\bar v_0)$. This ambiguity is tolerated since the points will tell us which space we are focusing on. It follows that $s(h'_+) \subseteq h'_+$ and $s(h'_-) \subseteq h'_-$. At least two of any three points must lie in the same half space corresponding to a given hyperplane $h'$, so without loss of generality we can assume $s(x'),s(y')\in h'_+$, which implies $\mu_{\Pi}(s(x'),s(y'),s(z')) \in h'_+$; on the other hand, according to the choice of orientation, we have $x',y' \in h'_+$, which implies $\mu_{\Pi'}(x',y',z') \in h'_+$, so $s(\mu_{\Pi'}(x',y',z')) \in h'_+$. This is a contradiction to our assumption (\ref{assumption1}), hence the claim is proved.

Returning to the proof we have $pr(s(\mu_{\Pi'}(x',y',z')))=pr(\mu_{\Pi}(s(x'),s(y'),s(z')))$, so either the two points $s(\mu_{\Pi'}(x',y',z')), \mu_{\Pi}(s(x'),s(y'),s(z'))$ are equal, or, by the analysis above, they are joined by an edge crossing $h$ in $\Pi$. It follows that
\[
\sigma(s(\mu_{\Pi'}(x',y',z'))) \thicksim_{\alpha(H(p))} \sigma(\mu_{\Pi}(s(x'),s(y'),s(z'))).
\]
Now for any $x',y',z'$, we have:
\begin{eqnarray*}
\sigma'(\mu_{\Pi'}(x',y',z')) &=& \sigma (s(\mu_{\Pi'}(x',y',z'))) \quad\thicksim_{\alpha(H(p))}\quad \sigma(\mu_{\Pi}(s(x'),s(y'),s(z')))\\
                              &\thicksim_{H(p)}&  \mu(\sigma(s(x')),\sigma(s(y')),\sigma(s(z')))  \quad=\quad \mu(\sigma'(x'),\sigma'(y'),\sigma'(z')),
\end{eqnarray*}
which means if we take $H'(p)=\beta(H(p))$ where $\beta(t)=\alpha(t)+t$, then
$$\sigma'(\mu_{\Pi'}(x',y',z')) \thicksim_{H'(p)} \mu(\sigma'(x'),\sigma'(y'),\sigma'(z')).$$
To sum up, if we have a 2-square in the original approximation $(\Pi,\pi,\sigma)$, then we can replace it by another approximation $(\Pi',\pi',\sigma')$. The controlling parameter $H'(p)=\beta(H(p))$ for the new approximation depends only on the original parameters $H(p),\rho,\kappav$ and the constant $\lambda$.
Since a CAT(0) cube complex generated by $p$ vertices has at most $2^p$ hyperplanes we can iterate this process at most $2^p$ times to remove all squares, ending with a new approximating tree $\Pi'$  with controlling parameter
\[
H'(p)=\beta^{2^p}(H(p))
\]
where the function $\beta$ defined above only depends on $\rho,\kappav$ and the constant $\lambda$.

This process defines a parameter $H'$ for which the approximating median algebras required by axiom (C2) can always be taken to be trees. Hence our space has rank $1$ with parameters $\rho, H'$.
\end{proof}

\subsection{Higher rank spaces}
In the proof above we deduced that the space $X$ satisfying the thin interval triangles condition (\ref{thintervaltriangles}) cannot possess arbitrary large ``coarse squares''. More explicitly, for a ``coarse square'' $a,b,c,d$ as in the proof, we showed that at least one of its parallel edge pairs must have bounded length. This intermediate result is crucial in our analysis of higher rank spaces, and it is exactly the rank 1 case of our general characterisation of rank. Intuitively, our characterisation states that a coarse median space has rank at most $n$ if and only if it does not contain arbitrarily large $(n+1)$-dimensional ``coarse cubes''.

\begin{thm}\label{char for high rank-final}
Let $(X,d,\mu)$ be a coarse median space and $n\in \mathbb{N}$. Then the following conditions are equivalent.
\begin{enumerate}
  \item $\rank X \leqslant n$;
  \item For any $\lambda>0$, there exists a constant $C=C(\lambda)$ such that for any $a,b\in X$, any $e_1,\ldots,e_{n+1}\in[a,b]$ with $\mu(e_i,a,e_j)\thicksim_\lambda a$ for all $i\neq j$, there exists $i$ such that $e_i\thicksim_C a$;
  \item For any $L>0$, there exists a constant $C=C(L)$ such that for any $L$-quasi-morphism $\sigma$ from the median $n$-cube $I^{n+1}$ to $X$, there exist adjacent points $x,y\in I^{n+1}$ such that $\sigma(x)\thicksim_C\sigma(y)$.
\end{enumerate}
\end{thm}

In condition (2) of the above theorem, one should imagine $a$ as a corner of an $(n+1)$-``coarse cube'', and $e_1,\ldots,e_{n+1}$ as endpoints of edges adjacent to $a$. 

\begin{proof}[Proof of Theorem \ref{char for high rank-final}]
\textbf{\emph{(1) $\Rightarrow$ (2):}} Let $\rho,H$ be parameters of $(X,d,\mu)$ under which $\rank X \leqslant n$ can be achieved. Given $\lambda>0$, $a,b\in X$ and $e_1,\ldots,e_{n+1}\in[a,b]$ with $\mu(e_i,a,e_j)\thicksim_\lambda a$ ($i\neq j$), we set $A=\{a,b,e_1,\ldots,e_{n+1}\}$. By axiom C2, there exists a finite rank $n$ median algebra $\Pi$, and maps $\pi \colon A \rightarrow \Pi$ and $\sigma \colon \Pi \rightarrow X$ satisfying the conditions in C2 and $\sigma \circ \pi =i_A$. Denote $\pi(a),\pi(b), \pi(e_i)$ by $\bar{a},\bar{b}, \bar{e}_i$. According to Lemma \ref{redefine cma on intvl}, without loss of generality, we can always take $\Pi=[\bar{a},\bar{b}]$ after changing the parameters within some controlled bounds if necessary. To make it clear, consider the following diagram:
\begin{displaymath}
\xymatrix@R=0.5cm{
                &  ~ &      X      \\
  [\bar{a},\bar{b}] \ar@/^/[r]^{inclusion} & \Pi \ar[ur]^{\sigma}  \ar@/^/[l]^{\mu_{\Pi}(\bar{a},\bar{b},\cdot)}                \\
                &  ~  &    A.  \ar[uu]^{i} \ar[ul]_{\pi}              }
\end{displaymath}
Define $\pi' \colon A \rightarrow [\bar{a},\bar{b}]$ by $x \mapsto \mu_{\Pi}(\bar{a},\bar{b},\pi{x})$, and $\sigma' \colon [\bar{a},\bar{b}] \rightarrow X$ by $\bar{c} \mapsto \sigma(\bar{c})$. Note that $\sigma'\circ\pi'(e_i)=\sigma(\mu_{\Pi}(\bar{a},\bar{b},\bar{e}_i))\thicksim_{H(n+3)} \mu(a,b,e_i)\thicksim_\kappav e_i$, since $e_i \in [a,b]$. For convenience, we still write $\pi,\sigma$ instead of $\pi',\sigma'$. To sum up, after modifying $\rho,H$ if necessary, we can always find a finite rank $n$ median algebra $\Pi$ and maps $\pi \colon A \rightarrow \Pi$ and $\sigma \colon \Pi \rightarrow X$ such that $\Pi=[\bar{a},\bar{b}]$, and conditions in (C2) hold (here we cannot require $\sigma \circ \pi =i_A$ again, but just close to $i_A$).

Since $[\bar{a}, \bar{b}]$ has rank at most $n$, without loss of generality, we can assume, by Lemma \ref{itrt median}, that
$$\mu_{\Pi}(\bar{e}_1,\ldots,\bar{e}_{n+1};\bar{b})=\mu_{\Pi}(\bar{e}_1,\ldots,\bar{e}_n;\bar{b}),$$
which implies that
$$\bigcap_{i=1}^{n+1} [\bar{e}_i,\bar{b}]=\bigcap_{i=1}^n [\bar{e}_i,\bar{b}].$$
So
$$\bigcap_{i=1}^n [\bar{e}_i,\bar{b}] \subseteq [\bar{e}_{n+1},\bar{b}], \text{~i.e.~} \mu(\bar{e}_1,\ldots,\bar{e}_n;\bar{b}) \in [\bar{e}_{n+1},\bar{b}].$$ Equivalently, $\bar{e}_{n+1} \in [\bar{a},\mu(\bar{e}_1,\ldots,\bar{e}_n;\bar{b})]$ since $\bar{e}_i\in [\bar{a},\bar{b}]$. By Lemma \ref{itrt eqn}, we have:
$$\bar{e}_{n+1}=\mu_\Pi(\bar{a},\bar{e}_{n+1},\mu(\bar{e}_1,\ldots,\bar{e}_n;\bar{b}))=\mu_\Pi(\mu_\Pi(\bar{a},\bar{e}_{n+1},\bar{e}_1),\ldots,\mu_\Pi(\bar{a},\bar{e}_{n+1},\bar{e}_n);\bar{b}).$$
Now translate the above equation into $X$, recall that $\sigma(\bar{e}_i)\thicksim_{H(n+3)}e_i$. By Lemma \ref{itrt C2}, there exists a constant $\alpha(n,\rho,H)$:
\begin{eqnarray*}
e_{n+1} &\thicksim_{H(n+3)} & \sigma(\bar{e}_{n+1})=\sigma(\mu_\Pi(\mu_\Pi(\bar{a},\bar{e}_{n+1},\bar{e}_1),\ldots,\mu_\Pi(\bar{a},\bar{e}_{n+1},\bar{e}_n);\bar{b}))\\
        &\thicksim_{\alpha(n,\rho,H)}& \mu(\sigma(\mu_\Pi(\bar{a},\bar{e}_{n+1},\bar{e}_1)),\ldots,\sigma(\mu_\Pi(\bar{a},\bar{e}_{n+1},\bar{e}_n));\sigma(\bar{b})).
\end{eqnarray*}
By (C1) and (C2), we know
$$\sigma(\mu_\Pi(\bar{a},\bar{e}_{n+1},\bar{e}_i))\thicksim_{H(n+3)}\mu(\sigma(\bar{a}),\sigma(\bar{e}_{n+1}),\sigma(\bar{e}_i)) \thicksim_{\rho(3H(n+3))} \mu(a,e_{n+1},e_i)\thicksim_\lambda a.$$
Now by Lemma \ref{itrt C1}, there exists a constant $\beta(\lambda,n,\rho,H)$ such that
$$\mu(\sigma(\mu_\Pi(\bar{a},\bar{e}_{n+1},\bar{e}_1)),\ldots,\sigma(\mu_\Pi(\bar{a},\bar{e}_{n+1},\bar{e}_n));\sigma(\bar{b})) \thicksim_{\beta(\lambda,n,\rho,H)} \mu(a,\ldots,a;b)=a,$$
which implies there exists some constant $C=C(\lambda)$ such that $e_{n+1}$ is $C$-close to $a$.

\textbf{\emph{(2) $\Rightarrow$ (3):}} Let $\rho,H$ be parameters of $(X,d,\mu)$. Let $\bar 0$ denote the zero vector in the median cube $(I^{n+1}, \mu_{n+1})$, let $\bar 1$ denote the vector with $1$ in all coordinates, and $\bar{e_i}$ denote the basis vector with a single $1$ in the $i$th coordinate. Given an $L$-quasi-morphism $\sigma: I^{n+1} \rightarrow X$, let $a=\sigma(\bar 0), b=\sigma(\bar 1)$ and $e_i=\sigma(\bar{e}_i)$. Since $\bar{e}_i \in [\bar 0,\bar 1]$, $\mu_{n+1}(\bar 0,\bar{e}_i,\bar 1)=\bar{e}_i$, which implies
$$e_i':=\mu(a,b,e_i)=\mu(\sigma(\bar 0),\sigma(\bar 1),\sigma(\bar{e}_i))\thicksim_{L}\sigma(\mu_{n+1}(\bar 0,\bar{e}_i,\bar 1))=e_i$$
and $e_i'\in[a,b]$. Now $\mu(e_i',a,e_j')=\mu(\mu(a,b,e_i),a,\mu(a,b,e_j))\thicksim_\kappav\mu(a,b,\mu(e_i,e_j,a))$, and
$$\mu(e_i,e_j,a)=\mu(\sigma(\bar{e}_i),\sigma(\bar{e}_j),\sigma(\bar 0)) \thicksim_{L} \sigma(\mu_{n+1}(\bar{e}_i,\bar{e}_j,\bar 0))=a.$$
So $\mu(e_i',a,e_j') \thicksim_{\rho(L)+\kappav} \mu(a,b,a)=a$. Now take $\lambda=\rho(L)+\kappav$, and by the assumption, there exists a constant $C'$ depending on $\lambda$ and hence implictly on $L$, such that one of $e_1',\ldots,e_{n+1}'$ is $C'$-close to $a$. This implies that  one of the points $\sigma(\bar{e_i})=e_i$ is $C$-close to $\sigma(\bar{0})=a$ for $C=C'+L$.

\textbf{\emph{(3) $\Rightarrow$ (1):}} Let $\rho,H$ be parameters of $(X,d,\mu)$. By definition, for any $p\in \mathbb{N}$ and $A \subseteq X$ with $|A|=p$, there exists a finite median algebra $(\Pi, \mu_{\Pi})$, and maps $\pi \colon A \rightarrow \Pi$, $\sigma \colon \Pi \rightarrow X$ satisfying axioms (C1), (C2), the conditions in Remark \ref{median assump} and $\sigma \pi = i_A$. Denote the associated finite CAT(0) cube complex also by $\Pi$. Now we need to modify $\Pi,\pi,\sigma$ to ensure that  $\Pi$ has rank $n$.

If $\rank \Pi=n$, nothing need to be modified. Assume $\rank \Pi \geqslant n+1$, so that, as a CAT(0) cube complex it has dimension at least $n+1$ and thus contains an $(n+1)$-cube. As a median subalgebra this is isomorphic to the median cube $(I^{n+1}, \mu_{n+1})$. By condition (3) with $L=H(p)$, we know there exists a constant $C$ depending on $L$ and hence on $p$ such that, without loss of generality, $\sigma(\bar{e}_{n+1})$ is $C$-close to $\sigma(\bar 0)$. Now construct the quotient CAT(0) cube complex $\Pi'$ from $\Pi$ by deleting the hyperplane $h$ crossing the edge $(\bar{0},\bar{e}_{n+1})$. Then the rest of the proof follows exactly that of Theorem \ref{char for rank1}.
\end{proof}

As mentioned in previous sections, the above theorem offers a way to define the rank of a coarse median space in the simplified setting. More explicitly, Theorem \ref{simp.cma} and Theorem \ref{char for high rank-final} combine to show
:
\begin{thm}\label{simp.cma.fnl}
Let $(X,d)$ be a metric space, and $\mu\colon X^3 \rightarrow X$ a ternary operation. Then $(X,d,\mu)$ is a coarse median space of rank at most $n$ if and only if the following  conditions hold:
\begin{itemize}
  \item[(M1).] $\mu(a,a,b)=a$ for any $a,b\in X$
  \item[(M2).] $\mu(a_{\sigma(1)}, a_{\sigma(2)},a_{\sigma(3)})=\mu(a_1,a_2,a_3)$, for any $a_1,a_2,a_3\in X$ and $\sigma$ a permutation;
  \item[(C1)'.] There exists an affine control function $\rho:[0,+\infty)\to [0,+\infty)$ such that for all $a,a',b,c\in X$,
      $$d(\mu(a,b,c), \mu(a',b,c)) \leqslant \rho(d(a,a'));$$
  \item[(C2)'.] There exists a constant $\kappaiv>0$ such that for any $a,b,c,d\in X$, we have
      $$\mu(\mu(a,b,c),b,d) \thicksim_{\kappaiv} \mu(a,b, \mu(c,b,d));$$
  \item[(C3)'.] $\forall\lambda>0, \exists C=C(\lambda)$ such that for any $a,b\in X$, any $e_1,\ldots,e_{n+1}\in[a,b]$ with $\mu(e_i,a,e_j)\thicksim_\lambda a$ for all $i\neq j$, there exists $i$ such that $e_i\thicksim_C a$.
\end{itemize}
\end{thm}

\section{A counterexample}\label{counterexample}
Recall that in a discrete median algebra, or equivalently, a CAT(0) cube complex, the interval between two points is the set of points which lie on geodesics connecting them. This makes a bridge between the algebraic aspect and the geometry of the object. It is natural to ask to what extent this holds in a coarse median space? As we have already seen in Theorem \ref{rank1}, it is ``almost'' true in rank $1$: intervals are ``about the same'' as the union of quasi-geodesics. As we shall see, the interaction of geodesics and intervals in higher rank is considerably less well behaved. We will show:

\begin{thm}
There is a rank $2$ geodesic coarse median space $(X,d,\mu)$ such that for any $C\geqslant 0$ there exist  $a,b\in X$, with the property that no geodesic $\gamma$ connecting $a,b$ lies within Hausdorff distance $C$ of the interval  $[a,b]$.
\end{thm}

\begin{proof}
Let $X=\mathbb{Z}^2$, and $\mu$ the canonical median operator on $X$, i.e.
$$\mu((x_1,y_1),(x_2,y_2),(x_3,y_3))=(m(x_1,x_2,x_3),m(y_1,y_2,y_3)),$$
where $(x_1,y_1),(x_2,y_2),(x_3,y_3) \in \mathbb{Z}^2$ and $m(\cdot,\cdot,\cdot)$ is the classical median of three real numbers. Equip $(X,\mu)$ with a metric $d$ as follows. We will view $X$ as a graph in the canonical way, and define $d$ to be a weighted edge-path metric on $X$, induced by assigning a positive number to each edge as its length. We define all horizontal edges to have length 1, i.e. $d((x,y),(x+1,y))=1$ for all $(x,y)\in \mathbb{Z}^2$; lengths of vertical edges are listed as follows. See Figure \ref{fig:def}.
\begin{figure}[htbp]
  \centering
  \includegraphics[width=10.5cm]{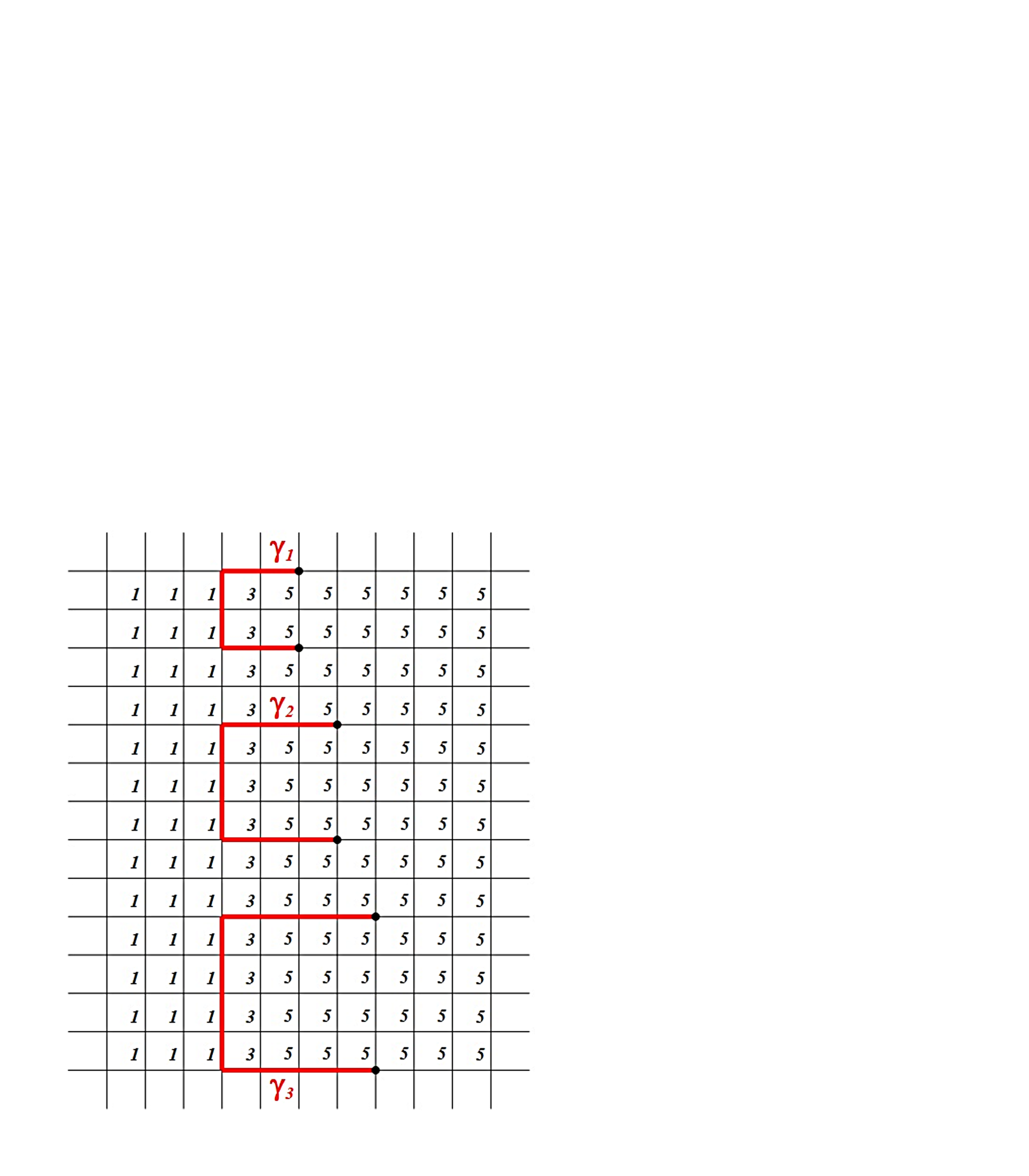}
  \caption{Lengths of vertical edges in the metric $d$ on $\mathbb{Z}^2$} \label{fig:def}
\end{figure}
\begin{itemize}
  \item If $x\leqslant 0$: $d((x,y),(x,y+1))=1$;
  \item If $x=1$: $d((x,y),(x,y+1))=3$;
  \item If $x \geqslant 2$: $d((x,y),(x,y+1))=5$.
\end{itemize}

Now consider a sequence $\{\gamma_n\}$ of edge paths in $(X,d)$: for each $n\in \mathbb{N}$, define a path starting from $a_n=(n+1,n+1)$, which travels horizontally leftwards to $(0,n+1)$, then vertically down to $(0,0)$ and finally horizontally rightwards to $b_n=(n+1,0)$. Again, see Figure \ref{fig:def}.

It is easy to check that the paths $\gamma_n$ are geodesic for all $n$. We claim that $(X,d,\mu)$ is a coarse median space. In fact, consider the identity map $id:(X,d_1,\mu) \rightarrow (X,d,\mu)$, where $d_1$ is the edge-path metric defined by each edge having length 1. It is obvious that this is a bi-Lipschitz map and a median morphism, so $(X,d,\mu)$ is indeed a coarse median space of rank 2. Now notice that $[a_n,b_n]=\{n+1\}\times[0,n+1]$. So $d_H(\mathrm{Im}\gamma_n,[a_n,b_n])=n+1$, which implies there is no uniform bound on the Hausdorff distance between geodesics and intervals.

Unfortunately the space we have constructed is not geodesic, for instance the points $(3,0), (3,1)$ are distance $5$ apart, but there is no (integer) path of length less than $7$ connecting them. We can modify the space easily to rectify this problem by subdividing the   edges of length $n>1$ by inserting $n-1$ points declared to be unit distance apart as appropriate. The natural extension of the metric $d$ to include these points is now geodesic and we continue to denote it by $d$. We are left with the issue of how to define coarse medians involving the inserted points. This is dealt with by projecting each point $p= (x,y)$ to its ``floor'' $\underline{p} =(x,\lfloor y\rfloor)$, the nearest original point to the  vertex at or below it on a vertical line. We then define the ternary operator $\mu'$ by $\mu'(p,q,r)=\mu(\underline{p},\underline{q},\underline{r})$ when $p,q,r$ are distinct and if two are equal then we define $\mu'(p,q,r)$ to be that point. Since the floor map moves points at most a distance of $4$ in the $d$ metric, $\mu'$ is still a coarse median operator with respect to the extended metric $d$, and intervals between $a_n$ and $b_n$ remain the same.
\end{proof}

\bibliographystyle{plain}
\bibliography{bibfileCMA}

\end{document}